\documentclass[10pt]{article}
\usepackage[dvips]{graphicx}
\usepackage{exscale,relsize}
\usepackage{bbm}
\usepackage{color}
\def\be{\begin{equation}}
\def\ee{\end{equation}}
\def\beq{\begin{eqnarray}}
\def\eeq{\end{eqnarray}}
\def\beqs{\begin{eqnarray*}}
\def\eeqs{\end{eqnarray*}}
\def\ea{\end{array}}
\def\ea{\end{array}}
\def\ds{\displaystyle}

\def\caR{{\cal R}}
\def\caV{{\cal V}}

\newcommand{\rline}  {{\RR}}

\newcommand{\half}   {{\frac{1}{2}}}

\newcommand{\nline}  {{\NN}}

\def\RR{{\rm I~\hspace{-1.15ex}R} }
\def\11{{\rm 1~\hspace{-1.5ex}1} }
\def\NN{{\rm I~\hspace{-1.15ex}N} }

\def\CC{\rm \hbox{C\kern-.56em\raise.4ex
         \hbox{$\scriptscriptstyle |$}\kern+0.5 em }}

\newcommand{\rfb}[1]{\mbox{\rm
   (\ref{#1})}\ifx\undefined\stillediting\else:\fbox{$#1$}\fi}

\makeatletter
\def\section{\@startsection {section}{1}{\z@}{-3.5ex plus -1ex minus
    -.2ex}{2.3ex plus .2ex}{\large\bf}}
\makeatother

\font\eufm=eufm10\font\eufms=eufm10\font\eufmss=eufm10\newfam\eufam
\textfont\eufam=\eufm\scriptfont\eufam=\eufms\scriptscriptfont\eufam=\eufmss

\newcommand{\eps}{\varepsilon}

\def\cO{\chi_{_{_{\lambda_0}}}}
\def\cOL{\chi_{_{_{\lambda_{0},L}}}}

\input{epsf}

\usepackage{amsthm}
\swapnumbers
\newtheorem{theorem}{Theorem}[section]
\newtheorem{lemma}[theorem]{Lemma}
\newtheorem{corollary}[theorem]{Corollary}
\newtheorem{remark}[theorem]{Remark}

\newtheorem{proposition}[theorem]{Proposition}
\newtheorem{Definition}[theorem]{Definition}

\begin{document}
\thispagestyle{empty}
\title{\bf Dispersive effects and high frequency behaviour for the Schr\"{o}dinger equation in
star-shaped networks}
\author{
Felix Ali Mehmeti $^\dag$, \, Ka\"{\i}s Ammari
\thanks{UR Analyse et Contr\^ole des Edp, UR13E564, D\'epartement de Math\'ematiques, Facult\'e des Sciences de
Monastir, Universit\'e de Monastir, 5019 Monastir, Tunisie, \ email: kais.ammari@fsm.rnu.tn}
, \, Serge Nicaise
\thanks{Universit\'e de Valenciennes et du Hainaut Cambr\'esis,
LAMAV, FR CNRS 2956, Le Mont Houy, 59313 Valenciennes Cedex 9,
France, \, email: felix.ali-mehmeti@univ-valenciennes.fr (F. Ali
Mehmeti), \, snicaise@univ-valenciennes.fr (S. Nicaise)}}
\date{}
\maketitle
%
%
%
{\bf Abstract.~} {\small We prove the time decay estimates $L^1
({\cal R}) \rightarrow L^\infty ({\cal R}),$ where ${\cal R}$ is an
infinite star-shaped network,
 for the Schr\"odinger group $e^{it(- \frac{d^2}{dx^2} + V)}$
for real-valued potentials $V$ satisfying some regularity
and decay assumptions. Further we show that the solution for initial conditions
with a lower cutoff frequency tends to the free solution,
if the cutoff frequency tends to infinity. \\


\section{Introduction} \label{formulare}
A characteristic feature of the Schr\"odinger equation is the loss
of the localization of wave packets during evolution, the
dispersion. This effect can be measured by $L^{\infty}$-time decay,
which implies a spreading out of the solutions, due to the time
invariance of the $L^{2}$-norm. The well known fact that the free
Schr\"odinger group in $\RR^n$ considered as an operator family from
$L^{1}$ to $L^{\infty}$ decays exactly as $c \cdot t^{-n/2}$ follows
easily from the explicit knowledge of the kernel of this group
\cite[p. 60]{ReedSimonII}.
For Schr\"odinger operators in one and three space dimensions with
potentials decaying sufficiently rapidly at infinity, similar
estimates have been proved in \cite{GS} for the projection of the
group on the subspace corresponding to the absolutely continuous
spectrum (without optimality). This approach uses an expansion in
generalized eigenfunctions together with estimates developed in
inverse scattering theory \cite{deift-trub}. We also refer to
\cite{wederb} for the Schr\"odinger equation on the half-line with
Dirichlet boundary conditions at 0.

In this paper we derive analogous $L^{\infty}$-time decay estimates
for Schr\"odinger equations with decaying potentials on a one
dimensional star shaped network. Further we state a perturbation
result showing that high energy solutions behave almost as free
solutions. For this purpose we furnish an explicit estimate of the
difference in terms of the lower cutoff frequency, the potential and
time. This result seems to be new even on the line.

Before a precise statement of our main results,
%
%
let us introduce some notation which will be used throughout the
rest of the paper.

Let $R_i,i=1,...,N,$ be $N (N \in \nline, N \geq 2)$ disjoint sets
identified with $(0,+\infty)$ and put ${\cal R} := \ds \cup_{k=1}^N
\overline{R}_k$. We denote by $f = (f_k)_{k=1,...,N} =
(f_1,...,f_N)$ the functions on ${\cal R}$ taking their values in
$\CC$ and let $f_k$ be the restriction of $f$ to $R_k$.

Define the Hilbert space ${\cal H} = \ds \prod_{k=1}^N L^2(R_k)$
with inner product $((u_k),(v_k))_{\cal H} = \ds \sum_{k=1}^N
(u_k,v_k)_{L^2(R_k)}$ and introduce the following transmission
conditions: \be \label{t0} (u_k)_{k=1,...,N} \in \ds \prod_{k=1}^N
C(\overline{R_k}) \; \hbox{satisfies} \; u_i(0) = u_k(0) \, \forall
\, i,k= 1,...,N, \ee

\be \label{t1}
(u_k)_{k=1,...,N} \in \ds \prod_{k=1}^N
C^1(\overline{R_k}) \; \hbox{satisfies} \; \ds \sum_{k=1}^N \frac{d
u_k}{dx}(0^+) = 0. \ee

Let $H_0 : {\cal D}(H_0) \rightarrow {\cal H}$ be the linear operator on
${\cal H}$ defined by :
$$
{\cal D}(H_0) = \left\{(u_k) \in \prod_{k=1}^N H^2(R_k); \, (u_k) \;
\hbox{satisfies} \; \rfb{t0},\rfb{t1} \right\},
$$
$$
H_0 (u_k) = (H_{0,k} u_k)_{k=1,...,N} = (- \frac{d^2 u_k}{dx^2})_{k=1,...,N} = - \Delta_{{\cal R}}(u_k).
$$
This operator $H_0$   is self-adjoint and
its spectrum $\sigma(H_0)$ is equal
 to $[0,+\infty)$ (see \cite{felix} for more details).

For any $s\in \RR$, let us denote by  $L^1_s ({\cal R})$ the space
of all complex-valued measurable functions
$\phi=(\phi_1,\ldots,\phi_N)$ defined on ${\cal R}$ such that
$$
\|\phi\|_{L^1_s ({\cal R})}:=\int_{\cal R} \left|\phi (x) \right|
\left\langle x \right\rangle^s \, dx =
\sum_{k=1}^N \int_{R_k} |\phi_k(x)|\left\langle x \right\rangle^s \,dx
< \infty,
$$
where $\langle x\rangle = (1+\left|x\right|^2)^{1/2}$.
This space is a Banach space with the norm $\|\cdot\|_{L^1_s ({\cal R})}$.

Let $V \in L^1_1 ({\cal R})$. Denote by $H$ the self-adjoint
realization of the operator $- \ds \frac{d^2}{dx^2} + V$  together
with the transmission conditions (\ref{t0}) and (\ref{t1}) on
$L^2({\cal R})$. From chapter 2 of \cite{berzin}, we deduce that its
spectrum satisfies
$$\sigma (H) = [0,+ \infty) \cup \left\{\hbox{a
finite number of negative eigenvalues} \right\}.$$
We first verify that the free Schr\"odinger group on the star-shaped network ${\cal R}$ satisfies the following dispersive
estimate (see Section \ref{free})
$$
\left\|e^{it H_0}\right\|_{L^1({\cal R}) \rightarrow L^\infty({\cal R})} \leq  C|t|^{- 1/2}, \, t \neq  0.
$$
Our goal is then to assume non restrictive  assumptions on the
potential $V$ in terms of decay or regularity in order to get a
similar decay for the Schr\"odinger equation with potential $V$.
More precisely, we will prove the following theorem.
\begin{theorem} \label{mainresult}
Let $V \in L^1_\gamma({\cal R})$ be real valued, with $\gamma>5/2$
and assume that (\ref{assumptionserge}) below holds. Then for all $t
\neq 0$, \be \label{dispest} \left\|e^{itH} P_{ac}(H)
\right\|_{L^1({\cal R}) \rightarrow L^\infty({\cal R})} \leq C \,
\left|t\right|^{- 1/2} \ee where $C$ is a positive constant and
$P_{ac}(H)$ is the projection onto the absolutely continuous
spectral subspace.
\end{theorem}

The assumption (\ref{assumptionserge})  is satisfied by a large
choice of potentials (see Lemma \ref{lassumptionserge} below). It
allows to built the kernel of the resolvent (see Definition
\ref{def.K} and Theorem \ref{theo1}) and takes into account the
ramification character of the problem.

At a first attempt, we have   assumed that $V \in L^1_\gamma({\cal
R})$, with $\gamma>5/2$ (in order to be able to apply some
appropriate estimates on the derivatives of the Jost functions, see
for instance Corollary \ref{corosdansH1}), while a similar result
probably holds under the assumption that  $V \in L^1_2({\cal R})$
(see \cite{GS} in the case $N=2$). This decay of the potential
implies that we do not need the so-called non resonance at zero
energy assumption (see \cite[pp. 163-164]{GS}). For potentials $V
\in L^1_1({\cal R})$ such an assumption would appear but it is a
difficult and delicate question. Furthermore, up to our knowledge,
if $V \in L^1_1({\cal R})$, it is unknown how to built  the kernel
of the resolvent.

As a consequence, we have the following $L^p - L^{p^\prime}$ estimate.
\begin{corollary} ($L^p - L^{p^\prime}$ estimate) \\
Under the assumptions of Theorem \ref{mainresult},   for $1 \leq p \leq 2$ and $\frac{1}{p} + \frac{1}{p^\prime} = 1$ we have for all $t \neq 0$,
\be
\label{strichartzbb}
\left\|e^{itH} P_{ac}(H)\right\|_{L^p({\cal R}) \rightarrow L^{p^\prime}({\cal R})} \leq C \, \left|t\right|^{-\frac{1}{p} + \half},
\ee
where $C > 0$ is a constant.
\end{corollary}

Moreover we have the following Strichartz estimates which have been used in the context of the nonlinear Schr\"odinger equation to obtain well-posedness results.

\begin{corollary} \label{stric} (Strichartz estimates)
Let  the assumptions of Theorem \ref{mainresult} be satisfied. Then for $2 \leq p,q \leq + \infty$ and $\frac{1}{p} + \frac{2}{q} = \frac{1}{2}$ we have for all $t$,
\be
\label{strichartzb1}
\left\|e^{itH} P_{ac}(H)f\right\|_{L^q(\rline,L^p({\cal R}))}  \leq C \, \left\|f\right\|_2, \, \forall \, f \in L^p({\cal R}) \cap L^2({\cal R}),
\ee
where $C > 0$ is a constant.
\end{corollary}

As a direct consequence, see \cite{casenave}, we have the following well-posedness result for a nonlinear Schr\"odinger equation with potential.
Let $p \in (0,4)$ and suppose that $V$ satisfies the assumptions of Theorem \ref{mainresult}. Then, for any $u_0 \in L^2({\cal R})$, there exists a unique solution
$$u \in C(\rline;L^2({\cal R})) \cap \; \ds \bigcap_{(q,r) \; \hbox{admissible}} L^q_{loc}(\rline;L^r({\cal R}))$$ of the equation
\be
\label{schb}
\left\{
\begin{array}{ll}
i u_t - \Delta_{{\cal R}} u  + V \, u \pm \left|u\right|^p u = 0, \, t \neq 0, \\
u(0) = u^0.
\end{array}
\right. \ee Recall that a pair $(q,r)$ is called  admissible  if
$(q,r)$ satisfies that $2 \leq r,q \leq +\infty$ and $\frac{2}{q} +
\frac{1}{r} = \frac{1}{2}$.

\begin{remark}
Another direct consequence of the dispersive estimate \rfb{dispest} or of the $L^p - L^{p^{\prime}}$  estimate \rfb{strichartzbb} is  that we can construct, as in \cite{weder}, the scattering operator for the nonlinear Schr\"odinger equation with potential.
\end{remark}


While proving Theorem \ref{mainresult} we obtain as results of
independent interest the $L^\infty-$time decay for the high
frequency part of the group and a high frequency perturbation
estimate:
\begin{theorem} \label{high frequency perturbation}
Under the assumptions of Theorem \ref{mainresult} we have
\begin{equation}\label{time decay}
\| e^{itH} \cO (H) \|_{1,\infty}
\leq (A + B \frac{\|V\|_1}{\sqrt{\lambda_0}})|t| ^{-1/2} ,
t \neq 0,
\end{equation}
\begin{equation}\label{perturbation}
\| e^{itH} \cO (H) -  e^{itH_0} \cO (H_0)\|_{1,\infty}
\leq B \frac{\|V\|_1}{\sqrt{\lambda_0}}|t| ^{-1/2} ,
t \neq 0 \, .
\end{equation}
Here $\cO$ is smoothly cutting off the frequencies below
$\lambda_0.$ Expressions of $A,B$ in terms of the cutoff function
but independent of $\lambda_0$ are given in Theorem \ref{high
frequency perturbation section}.

In particular we have for any $f\in L^1({\cal R})$ that
\[
e^{itH} \cO (H)f \rightarrow e^{itH_0} \cO (H_0)f \hbox{ for }
\lambda_0 \rightarrow \infty
\]
uniformly on ${\cal R}$ for every fixed $t>0$.

\end{theorem}
The perturbation estimate allows the simultaneous control of the
smallness of the difference between perturbed and unperturbed group
in terms of the cutoff frequency, the $L^1-$Norm of the potential
and time.

Because the reflection and refraction of wave packets for the
unperturbed Schr\"{o}dinger equation on the star shaped network is
known (\cite{admi2} for the case of 3 branches), the above
perturbation estimate furnishes an approximate spatial information
on the propagation of high frequency wave packets with explicit
control of the error. Note that the high frequency perturbation
estimate seems to be new even for the Schr\"{o}dinger equation with
potential on the line and represents in this case an improvement of
\cite{GS}. In \cite{GS}, estimate (\ref{time decay}) is furnished,
but without explicit control of the dependence of the coefficient of
$|t|^{-1/2}$ on $\lambda_0$. Without this control estimate
(\ref{perturbation}) is not useful to prove the convergence of the
solution to the free solution.


The paper is organized as follows. The second section deals with a
counterexample which shows that the decay of the Schr\"odinger
operator from $L^1({\cal R})$ to $L^\infty({\cal R})$ as $|t|$ goes
to infinity is not guaranteed for all infinite networks. In section
\ref{free}, we prove the dispersive estimate for the free
Schr\"odinger operator on star-shaped networks and we give some
direct applications. The expansion in generalized eigenfunctions
needed for the proof of Theorem \ref{mainresult}, is given in
section \ref{sec4}. In the last section we give the proof of the
main results of the paper (Theorems \ref{mainresult} and \ref{high
frequency perturbation}).


The main lines of our arguments are the following. The counter
example (section \ref{counterexample}) uses explicit formulas for
eigenfunctions of the laplacian on infinite trees from
\cite{Nicthesis}. The $L^{\infty}$-time decay of the free
Schr\"odinger group on a star shaped network is reduced to the
corresponding estimate on $\RR$ using an appropriate change of
variables (section \ref{free}). The task of finding a complete
family of generalized eigenfunctions for the Schr\"odinger operator
with potential on the star shaped network is reduced to the case of
the real line by separating the branches and  extending the
equations on $\RR$ with vanishing potential. The generalized
eigenfunctions on $\RR$ resulting from techniques from
\cite{deift-trub} are then combined to families on the network by
introducing correction terms to establish the transmission
conditions. Using results of \cite{deift-trub} for the real line
case, we derive estimates showing the dependence of the generalized
eigenfunctions on the potential. This enables us to prove a limiting
absorption principle and then to derive an expansion of the
Schr\"odinger group on the star in these generalized eigenfunctions
(section \ref{sec4}) following \cite{meh,FAMetall10}. The proof of
the $L^{\infty}$-time decay is divided in the low frequency and high
frequency part, essentialy following the lines of \cite{GS}. For the
high frequency components, the potential appears as a small
perturbation: the resolvent of the Schr\"odinger operator can be
expanded in a Neumann type series in terms of the resolvent of the
free Schr\"odinger operator. By inserting this in Stones formula and
exchanging the integration over the frequencies and the summation of
the Neumann series, one reduces the estimate to the free case. For
the low frequency components one uses the expansion in generalized
eigenfunctions derived in section \ref{sec4}, especially the
qualitative knowledge of the dependence of the generalized
eigenfunctions on the potential. This enables us to construct a
representation of the solution as the free Schr\"odinger group
acting on a well chosen (artificial) initial condition, which
encodes the influence of the potential. Then one concludes using the
results on the line.

Our approach does not furnish optimal results, as for example the
estimate in
\cite[p.~60]{ReedSimonII}
for the free Schr\"odinger group or the results of
\cite{L-infinity-proc-IWOTA}. This is due to the fact, that the use
of Neumann type series and qualitative estimates from inverse
scattering theory are to rough for this purpose. We conjecture that
optimal estimates could be achieved in terms of an asymptotic
expansion of first order following the lines of
\cite{L-infinity-proc-IWOTA}, where this problem has been solved for
initial conditions in energy bands for the Klein Gordon equation
with constant but different potentials on a star shaped network. It
might be useful to find a way to represent solutions for general
potentials by approximating these potentials by step functions,
inspired by \cite{borovskikh}.



Note that the general perturbation theory for semigroups
\cite[ch.~9, thm.~2.12, p.~502]{kato}
is applicable but not useful for our purposes: it yields that the
difference between the (semi-)groups generated by the
Schr\"{o}dinger operator with potential and the free one grows at
most proportionally to $t$, which engulfs the time decay at
infinity. Nevertheless it furnishes additional information for small
$t$.

The Trotter product formula \cite[thm.~X.51, p.~245]{ReedSimonII} is
also applicable, but cannot establish  $L^\infty-$time decay either:
it consists of an approximation of the perturbed group by long
alternating compositions of values of the free Schr\"{o}dinger group
$e^{it H_0}$ and the group of multiplication operators with
$e^{itV}$ but for small values of $t$. Thus even the explicit
knowledge of the kernel of the free Schr\"{o}dinger group is not
useful for time-decay, because the factor $t^{-1/2}$ becomes
effective only for large $t$.

The direct application of the variation of constants formula leads
to the same phenomenon as the perturbation for semigroups: without a
refined study of the superposition of the waves generated by the
potential, the rough estimation of the integral term leads to a
bound growing as a constant times $t$.

In \cite{banica} the authors prove dispersive estimates for
Schr\"{o}dinger equations on infinite trees with semi-infinite ends
with Kirchhoff conditions at the nodes. The equations do not have a
potential, but the operator has piecewise constant coefficients with
finitely many discontinuities on each branch. The coefficients are
bounded between two values. Here the difficulty comes from the
necessity to give a recursive formula for the infinitely many terms
of the resolvent of the operator. The inverse of the Wronskian is
estimated using the theory of almost periodic functions. In
\cite{banica:03}, \cite{IgnatZuazua} the authors study the
dispersion for the Schr\"{o}dinger equation on the line with
irregular coefficients.


In \cite{admi1} the authors consider Schr\"{o}dinger equations with
attractive cubic nonlinearities on a star-shaped network with three
branches. At the node they consider Kirchhoff-conditions, $\delta$-
or $\delta'$-conditions. They indicate that the equation arises in
quantum field theory, in the description of the Bose-Einstein
condensates and electromagnetic pulse propagation in optical fibers.
The Kirchhoff condition corresponds to a simple coupling ("beam
splitter"), whereas the $\delta-$condition describes the interaction
with a point-potential. The authors obtain charge and energy
conservation laws and deduce from these  facts conditions for global
in time existence of solutions. Further they treat the existence
life time of solitary waves and prove that their transmission and
reflection at the node is governed by the associated linear laws,
due to the shortness of the interaction time with a point-shaped
potential. However the authors do not consider variable potentials
on the branches as it is done in our paper. Therefore the linear
part of their paper has no substantial intersection with our setting
but might motivate further studies.

In \cite{admi2} an analogous setting as in \cite{admi1} is
considered, but with nonlinearities of order $2\mu + 1$ and only the
$\delta$-potential of strength $\alpha$ at the node. The existence
of stationary solitons in both the attractive ($\alpha<0$) and
repulsive ($\alpha>0$) case is proved. Again there is no significant
interference with our results.

In \cite{banicab} the authors consider free (linear) Schr\"{o}dinger
equations on tree-shaped networks with $\delta$-potentials at the
nodes. As a special case appears the star-shaped network with a
delta-potential at the center. In this setting a
$L^1-L^{\infty}$-decay estimate is proved. Due to the fact that the
$\delta$-potential plays the role of a transmission condition, the
methods are those for a problem with constant coefficients, and
therefore there is only a marginal interference with our results.
Nevertheless the result is instructive. The authors add the
existence and uniqueness of a global in time solution of the same
problem with a (attractive or repulsive) power nonlinearity of order
$p+1$.

The paper \cite{kost} deals with the general question of
constructing generalized eigenfunctions of all possible self adjoint
extensions of the Laplacian on networks with semi infinite ends. The
result is formulated in terms of a so called scattering matrix,
which indicates the reflected and transmitted flow for the
stationary problem. For complicated networks the authors construct a
product formula linking the scattering matrices of sub networks to
the scattering matrix of the original network. The results of this
article could serve to generalize our results to star shaped
networks with general transmission conditions.

The article \cite{nak} considers discrete analogs of nonlinear
Schr\"{o}dinger equations on star-shaped networks including the
existence of solitons, constants of motion and the calculus of
transmission probabilities.

Finally  \cite{sab} treats the stationary (cubic) nonlinear
Schr\"{o}dinger equation for simple but more general networks as the
star shaped ones as trees or helices. Explicit solution formulas are
obtained.

The last two papers are instructive for further developments of our
approach.

\noindent
{\bf Achnowledgements:}\\
The authors thank the referees for many valuable remarks which
helped us to improve the paper significantly.


\section{A counterexample} \label{counterexample}
Consider the infinite network $\caR= \ds \cup_{n\in \NN} e_n$, where each
edge $e_n=(n, n+1)$ with the set of vertices $\caV= \ds \cup_{n\in \NN}
v_n$, where $v_n=\{n\}$. For a fixed sequence of positive real
numbers $\alpha=(\alpha_n)_{n\in \NN}$, we define the Hilbert space
$L^2(\caR,\alpha)$ as follows
\[
L^2(\caR,\alpha)=\{u=(u_n)_{n\in \NN}: u_n\in L^2(e_n) \forall n\in
\NN \hbox{ such that } \sum_{n\in \NN}
\alpha_n\int_{e_n}|u_n(x)|^2\,dx <\infty\},
\]
equipped with the inner product
\[
(u,v)=\sum_{n\in \NN}\alpha_n\int_{e_n}  u_n(x)v_n(x)\,dx, \quad
\forall u,v\in L^2(\caR,\alpha).
\]
Similarly for all $k\in \NN^*$, we set
\[
H^k(\caR,\alpha)=\{u=(u_n)_{n\in \NN}\in L^2(\caR,\alpha):
(u_n^{(\ell)})_{n\in \NN}\in  L^2(\caR,\alpha) \ \forall\ell\in
\{1,2, \ldots, k\}\},
\]
where $u_n^{(\ell)}$ means the $\ell$ derivative of $u_n$ with
respect to $x$.

Now we consider the Laplace operator $-\Delta_\alpha$ (depending on
$\alpha$) as follows:
\[
{\cal D}(-\Delta_\alpha)=\{u=(u_n)_{n\in \NN}\in H^2(\caR,\alpha): \hbox{
satisfying } (\ref{H0}), (\ref{H1c}), (\ref{H1K}) \hbox{ below }\},
\]
\beq\label{H0} &&u_0(0)=0,\\
\label{H1c} &&u_n(n+1)=u_{n+1}(n+1), \forall n\in \NN,\\
\label{H1K} &&\alpha_n \frac{du_n}{dx}(n+1) =\alpha_{n+1} \frac{du_{n+1}}{dx}(n+1), \forall
n\in \NN. \eeq For all $u\in {\cal D}(-\Delta_\alpha)$, we set
\[
-\Delta_\alpha u=(- \frac{d^2u_n}{dx^2})_{n\in \NN}.
\]
By section 1.5 of \cite{Nicthesis}, this operator is a non negative
self-adjoint operator in   $L^2(\caR,\alpha)$.

Moreover in  Theorem 1.13 of \cite{Nicthesis} it was shown the
\begin{theorem}\label{tnic} For all  $k\in \NN^*$, $-k^2\pi^2$ is a simple
eigenvalue of $-\Delta_\alpha$ if and only if \be\label{condtnic}
s=\sum_{n\in \NN}\frac{1}{\alpha_n}<\infty. \ee In that case the
associated orthonormal eigenvector $\varphi^{[k]}=(\varphi^{[k]})_{n
\in \NN}$ is given by
\[
\varphi^{[k]}_{n}=\sqrt{\frac{2}{s}} \frac{(-1)^{(n-1)k}}{\alpha_n}
\sin (k\pi (x-n)), \forall x\in e_n, n\in \NN.
\]
\end{theorem}

Now assuming that (\ref{condtnic}) holds, then for any $k\in \NN^*$
we consider the solution $u$ of the Schr\"odinger equation
\[
\left\{ \begin{array}{ll}  \partial_t u-i\Delta_\alpha u=0 ,\\
u(t=0)=\varphi^{[k]},
\end{array}
\right.
\]
or equivalently solution of
\[
\left\{ \begin{tabular}{llll}  $\partial_t u_n-i \partial^2_x u_n=0 ,$&in &$e_n\times \RR$,\\
$u_0(0, t)=0,$ &on & $\RR$,\\
$u_n(n+1, t)=u_{n+1}(n+1, t)$&on & $\RR, \forall n\in \NN,$\\
$\alpha_n u'_n(n+1, t)=\alpha_{n+1}u'_{n+1}(n+1, t)$&on & $\RR, \forall n\in \NN,$\\
$u(t=0, \cdot)=\varphi^{[k]}$  &on &$\caR.$
\end{tabular}
\right.
\]
This solution is given by $u(t)=e^{-it k^2\pi^2} \varphi^{[k]}$.
Moreover  simple calculations show that \[ \|u(t)\|_{\infty,
\caR}=\sqrt{\frac{2}{s}}\sup_{n\in \NN} \frac{1}{\alpha_n}\|\sin
(k\pi (\cdot-n))\|_{\infty, e_n}=\sqrt{\frac{2}{s}}\sup_{n\in \NN}
\frac{1}{\alpha_n}, \] which is independent of $t$ and then does not
tend to zero as $|t|$ goes to infinity. On the other hand $u(t=0,
\cdot)$ belongs to $L^1(\caR)$, since  we have
\[ \|u(t)\|_{L^1(\caR)}=\sqrt{\frac{2}{s}}\sum_{n\in \NN} \frac{1}{\alpha_n}\|\sin
(k\pi (\cdot-n))\|_{L^1(e_n)}\leq \sqrt{{2}{s}}.
\]
In other words, we have proved the
\begin{theorem}\label{tnic2}
If (\ref{condtnic}) holds, then the  norm of the Schr\"odinger
operator $e^{it\Delta_\alpha}$ from $L^1(\caR)$ to $L^\infty(\caR)$
does not tend to zero as $|t|$ goes to infinity.
\end{theorem}

This counterexample shows that the decay of the  norm of the
Schr\"odinger operator   from $L^1({\cal R})$ to $L^\infty(\caR)$
as $|t|$ goes to infinity is not guaranteed for all infinite
networks. Hence the remainder the paper is to give some examples
where such a case occurs.

Let us notice that our non dispersive property comes from the
infinite numbers of discontinuities of the coefficient, since for a
finite number of discontinuities or BV coefficient with a  small
variation of the coefficients, the dispersive property holds, see
\cite{banica:03,IgnatZuazua}.

\section{Dispersive estimate for free Schr\"odinger operator on star-shaped networks} \label{free}

In this section we state the $L^\infty-$time decay estimate for the
free Schr\"{o}dinger equation (and some consequences) on   star
shaped networks. For completeness we give the proof, although it is
essentially the same as in \cite{admi1,ignat}.

\begin{theorem} (Dispersive estimate) \\
For all $t \neq 0$,
\be
\label{dispestf}
\left\|e^{itH_0} \right\|_{L^1({\cal R}) \rightarrow L^\infty({\cal R})} \leq C \, \left|t\right|^{- 1/2},
\ee
where $C > 0$ is a constant.
\end{theorem}
\begin{proof}
Let $v_j, \, j=1,..,.,N,$ a solution of the following problem
$$
\left\{
\begin{array}{ll}
\partial_t v_{j} = - i \partial^2_x v_{j}, \, \rline^+ \times \rline^+, \\
v_j(t,0) = v_1(t,0), \, \ds \sum_{j=1}^N \partial_x v_{j} (t,0) = 0, \, \rline^+,\\
v_j(0,x) = v_{j}^0(x), \, \rline^+.
\end{array}
\right.
$$
If we denote by $w_1 = \ds \sum_{j=1}^N v_j$ and $w_j = v_j - v_1, \, \forall \ , j=2,...,N.$

Then $w_1$ satisfies
$$
\left\{
\begin{array}{ll}
\partial_t w_{1} = - i \partial_x^2 w_{1}, \, \rline^+ \times \rline^+, \\
\partial_x w_{1}(t,0) = 0, \, \rline^+,\\
w_1(0,x) = \ds \sum_{j=1}^N v_{j}^0 (x), \, \rline^+,
\end{array}
\right.
$$

and $w_j, \, j=2,...,N,$ satisfies the following problem
$$
\left\{
\begin{array}{ll}
\partial_t w_{j} = - i \partial^2_x w_{j}, \, \rline^+ \times \rline^+, \\
w_{j}(t,0) = 0, \, \rline^+,\\
w_j(0,x) = v_{j}^0 (x) - v_{1}^0(x), \, \rline^+.
\end{array}
\right.
$$
By an odd reflection transformation applied to $w_1$, we obtain
$
\tilde{w}_1(t,x) = \left\{
\begin{array}{ll}
w_1(t,x), \, x > 0, \\
- \, w_1(t,- x), \, x < 0,
\end{array}
\right. $ which verifies
$$
\left\{
\begin{array}{ll}
\partial_t \tilde{w}_{1} = - i \, \partial^2_x \tilde{w}_{1}, \, \rline^2, \\
\tilde{w}_1(0,x) = \ds \sum_{j=1}^N \tilde{v}_{j}^0 (x), \, \rline,
\end{array}
\right.
$$
where $
\tilde{v}_{j}^0 = \left\{
\begin{array}{ll}
v_{j}^0(x), \, x > 0, \\
- \, v_{j}^0(- x), \, x < 0,
\end{array}
\right., \, j=1,...,N. $ So, according to the dispersive estimate
for Schr\"odinger operator on the line (see \cite{kato} or
\cite{ReedSimonII} for more details), we have \be
\label{estdispline} \left\| w_1 \right\|_{L^\infty(\rline^+)} \leq
\left\| \tilde{w}_1 \right\|_{L^\infty(\rline)} \leq C \, \left| t
\right|^{-\half} \, \left\| \ds \sum_{j=1}^N
\tilde{v}_{j}^0\right\|_{L^1(\rline)}, \, \forall \, (v_{j}^0) \in
L^2({\cal R}) \cap L^1({\cal R}), \ee where $C > 0$ is a constant.

Which implies
$$
\left\| w_1 \right\|_{L^\infty(\rline^+)} \leq 2 \, C \,
\left| t \right|^{-\half} \, \left\| \ds \sum_{j=1}^N v_{j}^0\right\|_{L^1(\rline^+)}, \, \forall \, (v_{j}^0) \in
L^2({\cal R}) \cap L^1({\cal R}).
$$
For $j=2,...,N$, we notice that  $w_j$ is solution of the free  Schr\"odinger equation on the half-line,
hence by Theorem 2.1 of \cite{wederb}, we get
\be
\label{estdisplineb}
\left\| w_j \right\|_{L^\infty(\rline_x^+)}  \leq  C \,
\left| t \right|^{-\half} \, \left\| v_{j}^0 -v_{1}^0 \right\|_{L^1(\rline^+)}, \, \forall \, (v_{j}^0) \in
L^2({\cal R}) \cap L^1({\cal R}),
\ee
where $C > 0$ is a constant.

Since, $v_j = w_j + v_1, \, \forall \, j=2,...,N$
and $v_1 + \ds \sum_{j=2}^N \left( w_j + v_1 \right) = w_1 \Rightarrow v_1 = \frac{1}{N} \, w_1 - \frac{1}{N} \, \ds \sum_{j=2}^N  w_j$.

Thus \rfb{estdispline}-\rfb{estdisplineb} imply that
\be
\label{estdisplinebb}
\left\| v_1 \right\|_{L^\infty(\rline^+)} \leq  \frac{4 C}{N} \,
\left| t \right|^{-\half} \, \ds \sum_{j=2}^N \left(\left\|v_{j}^0 \right\|_{L^1(\rline^+)} + \left\|v_{1}^0 \right\|_{L^1(\rline^+)} \right), \, \forall \, (v_{j}^0) \in
L^2({\cal R}) \cap L^1({\cal R}),
\ee
where $C > 0$ is a constant.

According to  the above we have
\be
\label{estdisplinebbv}
\left\| v_1 \right\|_{L^\infty(\rline^+)} \leq  4 C \,
\left| t \right|^{-\half} \, \ds \sum_{j=1}^N \left\|v_{j}^0 \right\|_{L^1(\rline^+)},
\ee
and
$$
\left\|v_j \right\|_{L^\infty(\rline^+)} \leq \left\|w_j \right\|_{L^\infty(\rline^+)}
+ \left\|v_1 \right\|_{L^\infty(\rline^+) } \leq  2C \, \left| t \right|^{-\half} \, \left( \left\|v_{j}^0 \right\|_{L^1(\rline^+)} + \left\|v_{1}^0 \right\|_{L^1(\rline^+)} \right) +
$$
\be
\label{estdisplinebbvv}
 4C \, \left|t\right|^{-\half} \, \ds \sum_{j=1}^N \left\|v_{j}^0 \right\|_{L^1(\rline^+)}, \, \forall \, (v_{j}^0) \in
L^2({\cal R}) \cap L^1({\cal R}),
\ee
$\Rightarrow$
\be
\label{estdisplinebbvvw}
\left\| v_j \right\|_{L^\infty(\rline^+)} \leq   8 \, C \,
\left| t \right|^{-\half}  \ds \sum_{j=1}^N \left\|v_{j}^0 \right\|_{L^1(\rline^+)}, \, \forall \, (v_{j}^0) \in
L^2({\cal R}) \cap L^1({\cal R}), \, \forall \, j \geq 2.
\ee
Finally we obtain for all $t \neq 0, \, (v_{j}^0) \in
L^2({\cal R}) \cap L^1({\cal R}),$
\be
\label{estdisplinebbvvwf}
\left\| (v_j) \right\|_{L^\infty({\cal R})} \leq  8 \, C \,
\left| t \right|^{-\half}  \ds \sum_{j=1}^N \left\|v_{j}^0 \right\|_{L^1(\rline^+)}= 8 \, C \,
\left| t \right|^{-\half}  \left\|(v_{j}^0) \right\|_{L^1({\cal R})},
\ee
which implies \rfb{dispestf}.
\end{proof}

As a direct consequence of the dispersive estimate for the free Schr\"odinger operator on a star-shaped network, we can obtain the following Strichartz estimates (for a direct proof, see \cite{ignat})
\begin{corollary} ($L^p - L^{p^\prime}$ estimate) \\
For $1 \leq p \leq 2$ and $\frac{1}{p} + \frac{1}{p^\prime} = 1$ we have for all $t \neq 0$,
\be
\label{strichartz}
\left\|e^{itH_0} \right\|_{L^p({\cal R}) \rightarrow L^{p^\prime}({\cal R})} \leq C \, \left|t\right|^{-\frac{1}{p} + \half},
\ee
where $C > 0$ is a constant.
\end{corollary}
\begin{proof}
According to \rfb{dispestf} we have
$$
\ds \sup_{t \neq 0} \left|t\right|^{\half} \, \left\|e^{itH_0} f \right\|_\infty \leq C \, \left\|f \right\|_{1}, \, \forall \, f \in L^1({\cal R}) \cap L^2({\cal R}).
$$
Interpolating with the $L^2$ bound $\left\|e^{itH_0} f \right\|_{2} = \left\|f \right\|_{2},$ leads to
\be
\label{pq}
\ds \sup_{t \neq 0} \left|t \right|^{- \half + \frac{1}{p}} \left\|e^{itH_0} f \right\|_{p^\prime} \leq C \, \left\|f \right\|_p, \, \forall \, f \in L^1({\cal R}) \cap L^2({\cal R}),
\ee
where $ 1 \leq p \leq 2$. It is well-known that via $T^*T$ argument \rfb{pq} gives rise to the class of Strichartz estimates
\be
\label{stricf}
\left\|e^{itH_0} f \right\|_{L^{q}_t (L^{p}_x)} \leq C \, \left\| f\right\|_2, \; \forall \, \frac{2}{q} + \frac{1}{p} = \frac{1}{2}, \, 2 < q \leq + \infty, \, 2 \leq p \leq \infty.
\ee
The endpoint $q=2$ is not captured by this approach but by the approach develloped by Keel and Tao in \cite{keel}. So the estimate \rfb{stricf} is valid for all $2 \leq p,q \leq + \infty$ satisfying $\frac{2}{q} + \frac{1}{p} = \frac{1}{2}$
and we have also,
$$
\left\|\int_\rline e^{-itH_0} F(s,.) ds \right\|_{L^2({\cal R})} \leq C \, \left\|  F\right\|_{L^{q^{\prime}}
(\rline,L^{p^{\prime}}({\cal R}))},
$$

$$
\left\|\int_0^t e^{i(t-s)H_0} F(s) ds \right\|_{L^q(\rline, L^{r^\prime}({\cal R}))} \leq C \, \left\|  F\right\|_{L^{r^{\prime}}
(\rline,L^{s^{\prime}}({\cal R}))},
$$
for all admissible pairs $(q,p)$ and $(r,s)$ satisfying $\frac{2}{q} + \frac{1}{p} = \frac{1}{2}, \, 2 \leq q,p \leq + \infty$. \end{proof}

Corollary \ref{stric}   can be proved in the same way.


According to \rfb{stricf} and \cite{casenave}, we have for $p \in (0,4),$ that for any $u_0 \in L^2({\cal R})$ the equation

$$
i u_t - \Delta_{{\cal R}} u  \pm |u|^p u =0, \, t \neq 0, \;
u = u_0, \, t = 0,
$$
admits a unique solution $u \in C(\rline,L^2({\cal R})) \cap \, \ds \bigcap_{(q,r) \; \hbox{admissible}} L^q_{loc}(\rline,L^r({\cal R})).$

For similar results about nonlinear Schr\"odinger equation on graphs, we refer to \cite{admi2,nak,sab}.

\section{Expansion in generalized eigenfunctions} \label{sec4}
The goal of this section is to find an explicit expression for the
kernel of the resolvent of the operator $H$ on the star-shaped
network defined in section \ref{formulare}. First we separate the
branches by extending the potential of the Schr\"{o}dinger operator
by zero on $(-\infty,0)$. Using \cite{deift-trub}, we construct $N$
families of generalized eigenfunctions of the resulting $N$
Schr\"{o}dinger operators on $\RR$, which we recombine on the
network.  This approach can be compared with the ones developed for
Klein-Gordon equations in ${\cal R}$ by \cite{meh,FAMetall10}.

For each $j=1,\ldots, N$, we recall that $R_j$ is identified to
$(0,+\infty)$ and denote by $V_j$ the restriction of $V$ to $R_j$.
Consider $R_j$ as a subset of $\RR$ and denote by $\tilde V_j$ the
extension of $V_j$ by 0 outside $R_j$.

Now according to  \cite{deift-trub}  (see also \cite{weder,wederb}) for all
$z\in \CC^+:=\{z_1\in \CC: \Im z_1\geq 0\}$, there exist two
functions $f_{j,\pm} (z,\cdot)$ that satisfy the differential
equation
\be
\label{josteq} - f^{\prime \prime}_{j,\pm} (z,x) + \tilde V_j(x)
f_{j,\pm} (z,x)  = z^2 f_{j,\pm} (z,x) \hbox{ on } \RR,
\ee
and that have the asymptotic behaviour
\be \label{asymp.osc} |f_{j,\pm} (z,x)-e^{\pm izx}|\to 0 \hbox{ as }
x\to \pm \infty. \ee
According to section 1 of \cite{deift-trub} (see also \cite[p.
45]{weder}) we write
$$
f_{j,\pm} (z,x)=e^{\pm izx}  m_{j,\pm} (z,x),
$$
to remove the oscillations of $f_{j,\pm}$ at infinity.
The functions $m_{j,\pm}$ are the unique solutions of the Volterra
integral equations:
\beq
\label{mj+} m_{j,+}(z,x)&=&1+\int_x^{+\infty}
\frac{e^{2iz(y-x)}-1}{2iz} \tilde V_j(y) m_{j,+}(z,y) \, dy,
\\
m_{j,-}(z,x)&=&1+\int_{-\infty}^x \frac{e^{2iz(y-x)}-1}{2iz} \tilde V_j(y) m_{j,-}(z,y) \, dy,
\label{mj-}
\eeq
and are called Jost functions (see \cite{deift-trub,ReedSimonIII}).
Recall that   Lemma 1 of \cite{deift-trub} (see also (2.5) of
\cite{weder})   implies that \beq\label{estm+} |m_{j,+}(z,x)|\leq C,
\quad \forall x\in [0,\infty), z\in \CC^+,
\\
|m_{j,-}(z,x)|\leq 1+C \frac{1+x}{1+|z|}, \quad \forall x\in [0,\infty), z\in \CC^+,
\label{estm-}
\eeq
for some $C>0$. Accordingly as
$
f_{j,\pm} (z,x)=e^{\pm izx}  m_{j,\pm} (z,x),
$
we get
\beq\label{estf+}
|f_{j,+}(z,x)|\leq C, \quad \forall x\in [0,\infty), z\in \CC^+,
\\
|f_{j,-}(z,x)|\leq  C (1+x) e^{\Im z x}, \quad \forall x\in [0,\infty), z\in \CC^+.
\label{estf-}
\eeq

Property \rfb{asymp.osc} implies the existence of functions
$T_j, R_{j,1}, R_{j,2}, j=1,\ldots,N$, called transmission and reflection
coefficients,
such that
\begin{eqnarray*}
  f_{j,+}(x,r) &\sim& \frac{1}{T_j(r)}e^{irx} + \frac{R_{j,2}(r)}{T_{j}(r)}e^{-irx},
  \ x \rightarrow -\infty\\
 f_{j,-}(x,r) &\sim& \frac{1}{T_j(r)}e^{-irx} + \frac{R_{j,1}(r)}{T_{j}(r)}^{irx},
  \ x \rightarrow \infty\\
\end{eqnarray*}
for $r \in \RR $.
%
For future purposes, for all real numbers $r$, we need  the
scattering matrix $S_j(r)\in \CC^{2\times 2}$ associated with
(\ref{josteq})   given by
$$
S_j(r)=\left(\begin{array}{ll}
T_j(r)& R_{j,2}(r)\\
R_{j,1}(r)&T_j(r)
\end{array}
\right)
$$
and that is continuous on $\RR$.
According to \cite{deift-trub}, $T_j$ has a meromorphic extension to $\CC^+$ (with a finite numbers of simple poles that are non zero purely imaginary numbers) that is given by (see  \cite[p. 145]{deift-trub})
\be\label{expTj}
\frac{1}{T_j(z)}=1-\frac{1}{2iz}\int_{-\infty}^{+\infty}  \tilde V_j(y) m_{j,+}(z,y) \, dy\quad \forall z\in \CC^+.\ee
Since $\tilde V_j$ has its support in $(0,+\infty)$,  by remark 10 of \cite{deift-trub}
$R_{j,2}$ admits also a meromorphic extension on  $\CC^+\setminus\RR$ (with the same poles as the ones of $T_j$)
that is given by (compare  \cite[p. 145]{deift-trub} when $z$ is real)
\be\label{expR2j}
\frac{R_{j,2}(z)}{T_j(z)}=\frac{1}{2iz}\int_{-\infty}^{+\infty}  e^{2iz y}\tilde V_j(y) m_{j,+}(z,y) \, dy\quad \forall z\in \CC^+.\ee

Due to the fact that $\tilde V_j$ is zero on $(-\infty, 0)$, the
generalized eigenfunctions $f_{j,\pm}$ of the Schr\"{o}dinger
operators on the line have the following properties.

\begin{lemma}\label{lpropJosti}
For all $z \in \CC^+,$ $z \ne 0$, we have
\beq
\label{propJosti-}
f_{j,-} (z,x)&=&e^{-izx} \quad \forall x\leq 0,
\\
f_{j,+} (z,x)&=&\frac{1}{T_j(z)}e^{izx}+\frac{R_{j,2}(z)}{T_j(z)}e^{-izx} \quad \forall x\leq 0.
\label{propJosti+}
\eeq
In particular,
it holds
\beq\label{propJosti-0}
f_{j,-} (z,0)&=&1,
\\
f_{j,+} (z,0)&=&\frac{1+R_{j,2}(z)}{T_j(z)}.
\label{propJosti+0}
\eeq
\end{lemma}
\begin{proof}
From the expression (\ref{mj-}), we directly get (\ref{propJosti-})
and  (\ref{propJosti-0}). The situation  is more complicated for
$f_{j,+} $. Indeed from the   expression (\ref{mj+}), we see that
$$
m_{j,+}(z,x)=1+\int_0^{+\infty} \frac{e^{2iz(y-x)}-1}{2iz}  \tilde V_j(y) m_{j,+}(z,y) \, dy, \forall x\leq 0.
$$
This is equivalent to
\beqs
m_{j,+}(z,x)&=&1-\frac{1}{2iz}\int_0^{+\infty}  \tilde V_j(y) m_{j,+}(z,y) \, dy+\frac{e^{-2izx}}{2iz}\int_0^{+\infty} e^{2iz y}  \tilde V_j(y) m_{j,+}(z,y) \, dy
\\
&=&1-\frac{1}{2iz}\int_{-\infty}^{+\infty}  \tilde V_j(y) m_{j,+}(z,y) \, dy+\frac{e^{-2izx}}{2iz}\int_{-\infty}^{+\infty} e^{2iz y}  \tilde V_j(y) m_{j,+}(z,y) \, dy, \forall x\leq 0.
\eeqs
Hence according to the expression of $\frac{1}{T_j(z)}$ and $\frac{ R_{j,2}(z)}{T_j(z)}$
given in (\ref{expTj}) and (\ref{expR2j}), we obtain  (\ref{propJosti+}). According to this identity we trivially have
$$
f_{j,+} (z,0)=\frac{1+R_{j,2}(z)}{T_j(z)}\ .
$$ \end{proof}

For our next considerations, we need that
$$
f_{j,+} (z,0)  \ne 0,
$$
at least for all $z\in \CC^+$ close to the real axis.

Therefore we make the following assumption:
\be
\label{assumptionserge} 1+\int_0^{+\infty} x V_j(x)
m_{j,+}(0,x)\,dx\ne 0, \forall j=1,\ldots, N,
\ee
that allows to obtain the next result.
\begin{lemma}\label{leigenfcts+nonnulle}
If the assumption (\ref{assumptionserge}) holds, then there exists
$\kappa>0$ small enough and two positive constants $C_1, C_2$ such
that \be\label{bornitude} C_1\leq |f_{j,+} (z,0)|\leq C_2 \quad
\forall z\in B_\kappa, \ee where
$B_\kappa=\{z_1\in \CC^+: 0\leq \Im z_1\leq \kappa\}$.
\end{lemma}
\begin{proof}
Recall that
$$
f_{j,+} (z,0)=\frac{1+R_{j,2}(z)}{T_j(z)}.
$$
By (\ref{expTj}) and (\ref{expR2j}) we see that (see property IV of
Theorem 1 in \cite{deift-trub}, p.~147) there exist $R, C>0$ such
that \be\label{bornitude1} |T_j(z)-1|+|R_{j,2}(z)|\leq
\frac{C}{|z|}, \forall |z|>R. \ee Hence (\ref{bornitude}) holds for
all $|z|>R_0$, with $R_0$ large enough.

Now for $|z|$ small, we remark that $\frac{1+R_{j,2}(z)}{T_j(z)}$ is different from zero for all
$z\in \RR\setminus\{0\}$
by using the properties II and V of
Theorem 1 in \cite[p. 146]{deift-trub}.
Furthermore using (\ref{expTj}) and (\ref{expR2j}), one easily checks that
\be\label{serge10}
\lim_{z\to 0}\frac{1+R_{j,2}(z)}{T_j(z)}=1+\int_0^{+\infty} t V_j(t) m_{j,+}(0,t)\,dt.
\ee
Consequently our assumption garantess that the continuous function
$f_{j,+} (\cdot,0)$ is different fom zero on the whole compact $[-R_0, R_0]$
and therefore
 (\ref{bornitude}) holds for all real numbers $z\in [-R_0, R_0]$.
By the continuity of $f_{j,+} (\cdot,0)$ on $B_{\delta'}$ for $\delta'$ small enough,
we deduce that  (\ref{bornitude}) holds for all $z\in B_{\kappa}\cap \{z_1\in \CC: \Re z_1\in[-R_0, R_0]\}$,
by choosing $\kappa$ small enough.
\end{proof}

The assumption (\ref{assumptionserge}) is technical but it is satified
by a large choice of potentials.
Let us list some specific examples.

\begin{lemma}\label{lassumptionserge}
1. In the generic case, namely
if
$$
\int_0^{+\infty} V_j(x) m_{j,+}(0,x)\,dx\ne 0,
$$
then we have
\be
\label{assumptionsergepart} 1+\int_0^{+\infty} x V_j(x)
m_{j,+}(0,x)\,dx\ne 0,
\ee
if $V_j$ is non negative or if
$$
\int_0^{+\infty} x |V_j(x)|\,dx\leq \rho
$$
where $\rho$ is the unique positive number such that $\rho e^\rho=1$.
\\
2. In the exceptional case,
namely
if
$$
\int_0^{+\infty} V_j(x) m_{j,+}(0,x)\,dx=0,
$$
then (\ref{assumptionsergepart}) always holds.
\end{lemma}
\begin{proof}
In the exceptional case, by Theorem 1 of \cite{deift-trub}, there exists a  constant $C\in (0,1)$
such that
$$
|R_{j,2}(r)|\leq C, \forall r\in \RR.
$$
Hence
$$
\lim_{r\to 0\atop r\in \RR}\left|\frac{1+R_{j,2}(r)}{T_j(r)}\right|\geq  1-C,
$$
which implies that (\ref{assumptionsergepart})  holds.

In the generic case and if $V_j$ is non negative, then $m_{j,+}(0,\cdot)$ is a non negative
  function and therefore (\ref{assumptionsergepart}) directly  holds.

In the generic case and if $V_j$ has no sign, then the considerations of Lemma  1 of \cite[p. 133]{deift-trub}
shows that
$$
|m_{j,+}(0,0)|\geq 1-\gamma_je^{\gamma_j},
$$
where $\gamma_j=\int_0^{+\infty} t |V_j(t)|\,dt$. Hence if $1-\gamma_je^{\gamma_j}>0$,
we deduce that
$m_{j,+}(0,0)$ is different from zero.
This yields the conclusion since
$$
m_{j,+}(0,0)=f_{j,+}(0,0)=\lim_{z\to 0}f_{j,+} (z,0).
$$
\end{proof}

Note that   $V_j=0$ is an  exceptional case.

We now prove that $R_{j,2}(z)$ is continuous and uniformly bounded in $B_\kappa$
if $\kappa>0$ small enough (suggested by Remark 10 of \cite{deift-trub}).
\begin{lemma}\label{lrj2borne}
For all $j=1,\ldots, N$,
there exists a positive constant $C_j$ such that
\be\label{serge30}
|R_{j,2}(z)|\leq C_j, \quad  \forall z\in B_\kappa,
\ee
for $\kappa>0$ small enough.
\end{lemma}
\begin{proof}
By Theorem 1 of \cite{deift-trub},
there exists $C_1>0$ such that
$$
|T_j(z)|\leq C_1, \forall z\in B_\kappa,
$$
for $\kappa>0$ small enough.
Hence by (\ref{expR2j}) we deduce that (\ref{serge30}) holds
for all $|z|>\epsilon$, for any $\epsilon>0$.

For $z$ in the ball  $|z|\leq \epsilon$, we distinguish the generic case from the exceptional one.
In the generic case, by part V of Theorem 1 of \cite[p. 150]{deift-trub},
we know that
$$
T_j(z)=\alpha_j z+o(z), \hbox{ for } z \rightarrow 0
$$
with $\alpha_j\ne 0$
and again using (\ref{expR2j}) we deduce that (\ref{serge30}) for $|z|\leq \epsilon$.

In the exceptional case, by (\ref{expR2j}) we may write

$$
R_{j,2}(z)=\frac{T_j(z)}{2iz}
\left(\int_{0}^{+\infty}  (e^{2iz y}-1) V_j(y) m_{j,+}(z,y) \, dy
+\int_{0}^{+\infty}   V_j(y) (m_{j,+}(z,y)-m_{j,+}(0,y)) \, dy\right),
$$
because $\int_0^{+\infty} V_j(t) m_{j,+}(0,t)\,dt=0$.
Therefore we obtain that
$$
|R_{j,2}(z)|\leq C_1
\left(\left|\int_{0}^{+\infty}  \frac{e^{2iz y}-1}{2iz} V_j(y) m_{j,+}(z,y) \, dy\right|
+\left|\int_{0}^{+\infty}   V_j(y) \frac{m_{j,+}(z,y)-m_{j,+}(0,y)}{2iz} \, dy\right|\right).
$$
For the first term of this right hand side, due to (\ref{estm+})
we can directly apply the dominated convergence theorem
to conclude that
$$
\int_{0}^{+\infty}  \frac{e^{2iz y}-1}{2iz} V_j(y) m_{j,+}(z,y) \, dy\to \int_{0}^{+\infty}  yV_j(y) m_{j,+}(0,y) \, dy\quad \hbox{ as } z\to 0.
$$
Since this limit is finite, we deduce that
$$
\left|\int_{0}^{+\infty}  \frac{e^{2iz y}-1}{2iz} V_j(y) m_{j,+}(z,y) \, dy\right|\leq C,
$$
for $|z|$ small enough.

For the second term, we use the same argument.
Namely
since $\tilde V_j$ belongs to $L^1_2(\RR)$, by Remark 3 of \cite{deift-trub},
the derivative $\dot{m}_{k,+}$ of $m_{k,+}$ with respect to $k$ exists  and is continuous on $\CC^+$.
Moreover by Lemma 2.1 of \cite[p. 46]{weder}, there exists $C_2>0$ such that
\be\label{estmdot}|\dot{m}_{k,+}(z,y)|\leq C_2, \forall x\geq 0.
\ee

Consequently by using the mean value theorem we have
$$
\frac{m_{j,+}(z,y)-m_{j,+}(0,y)}{2iz}=\frac{\dot{m}_{k,+}(\theta z,y)}{2i},
$$
for some $\theta \in (0,1)$
and therefore
$$
\left|\frac{m_{j,+}(z,y)-m_{j,+}(0,y)}{2iz}\right|\leq \frac{C_2}{2}, \forall x\geq 0.
$$
The application of dominated convergence theorem
yields
$$
  \int_{0}^{+\infty}   V_j(y) \frac{m_{j,+}(z,y)-m_{j,+}(0,y)}{2iz} \, dy\to \int_{0}^{+\infty}  V_j(y) \dot{m}_{j,+}(0,y) \, dy \quad \hbox{ as } z\to 0.
$$
The conclusion follows since this right-hand side is finite.
\end{proof}

We are now ready to give the different families of generalized eigenfunctions of $H$.

\begin{lemma}\label{lgeneigenfcts}
Under the assumption (\ref{assumptionserge}), then
for all $z\in B_\kappa$, $z\ne 0$  and all $j \in \{ 1, \cdots, N\}$, there exist two generalized eigenfunctions $F_{z^2}^{\pm,j}: {\cal R} \rightarrow \CC$ of $H$ defined by
$$
F_{z^2}^{\pm,j}(x):=F_{z^2,k}^{\pm,j}(x)\quad\forall x \in \overline{R_k},
$$
where $F_{z^2,k}^{\pm,j}$ is in the form
\be\label{geneigenfcts} \left\{ \begin{array}{ccc}
       F_{z^2,j}^{\pm,j}(x) &=& c_{j,\pm,1}(z) f_{j,\pm}(z,x)+c_{j,\pm,2}(z) f_{j,\mp}(z,x),   \\
      F_{z^2,k}^{\pm,j}(x) &=& d_{j,k,\pm}(z) f_{k,\mp}(z,x), \forall k \neq j,
   \end{array} \right.
\ee
and $c_{j,\pm,1}(z)$, $c_{j,\pm,2}(z)$and $d_{j,k,\pm}(z)$ are given by (modulo $N$)
\beqs
c_{j,\pm,1}(z)&=&\frac{ f_{j+1,\mp}(z,0)}{W_{j,\pm}(z)}
\left( f'_{j,\mp}(z,0)+f_{j,\mp}(z,0) \sum_{k\ne j}
 \frac{f'_{k,\mp}(z,0)}{f_{k,\mp}(z,0)}\right),\\
c_{j,\pm,2}(z)&=&-\frac{f_{j+1,\mp}(z,0)}{W_{j,\pm}(z)}
\left(   f'_{j,\pm}(z,0)+f_{j,\pm}(z,0)\sum_{k\ne j}
\ \frac{f'_{k,\mp}(z,0)}{f_{k,\mp}(z,0)}\right),\\
d_{j,k,\pm}(z)&=&\frac{f_{j+1,\mp}(z,0)}{f_{k,\mp}(z,0)}, \forall k\ne j,
\eeqs
$W_{j,\pm}(z)$ is the Wronskian  relatively to $f_{j,\pm}$, namely
$$
W_{j,\pm}(z)=f_{j,\pm}(z,x)f'_{j,\mp}(z,x)-f_{j,\mp}(z,x)f'_{j,\pm}(z,x),
$$
that is constant in $x$ and different from 0 (since $z\ne 0$).
\end{lemma}
\begin{proof}
We look for generalized eigenfunctions in the form (\ref{geneigenfcts}), the constants
$c_{j,\pm,1}(z)$, $c_{j,\pm,2}(z)$ and $d_{j,k,\pm}(z)$ will be fixed below in order to guarantee the continuity of $F_{z^2}^{\pm,j}$
at $0$
and the Kirchoff law. This will show that $F_{z^2}^{\pm,j}$ are generalized eigenfunctions of $H$
since   $F_{z^2,k}^{\pm,j}$ satisfies
$$
- \frac{d^2}{dx^2} F_{z^2,k}^{\pm,j} (x) + \tilde V_j(x) F_{z^2,k}^{\pm,j} (z,x)  = z^2 F_{z^2,k}^{\pm,j}\hbox{ on } R_k.
$$

Since each branch $j$ plays the same rule, we can take $j=1$
and write $c_{1,\pm,1}(z)=c_1$, $c_{1,\pm,2}(z)=c_2$ and $d_{1,k,\pm}(z)=d_{k}$.
The continuity at 0 is  equivalent to
$$
c_{1} f_{1,\pm}(z,0)+c_{2} f_{1,\mp}(z,0)= d_{k} f_{k,\mp}(z,0)\quad \forall k\ne 1,
$$
while the Kirchoff law
is equivalent to
$$
c_{1} f'_{1,\pm}(z,0)+c_{2} f'_{1,\mp}(z,0)+
\sum_{k=2}^N   d_{k} f'_{k,\mp}(z,0)=0.
$$

Since by Lemma \ref{lpropJosti} $f_{k,\mp}(z,0)$ is different from 0,
we will get
$$
d_{k}=\frac{d_{2}f_{2,\mp}(z,0)}{f_{k,\mp}(z,0)}, \forall k\ne 1,
$$
and the continuity and the Kirchoff law reduce to
$$
\left\{
\begin{array}{ll}
c_{1} f_{1,\pm}(z,0)+c_{2} f_{1,\mp}(z,0)= d_2 f_{2,\mp}(z,0),\\
 c_{1} f'_{1,\pm}(z,0)+c_{2} f'_{1,\mp}(z,0)=
-  d_{2}f_{2,\mp}(z,0)  \sum_{k=2}^N   \frac{f'_{k,\mp}(z,0)}{f_{k,\mp}(z,0)}.
\end{array}
\right.
$$
This $2\times 2$ linear system in $c_1$ and $c_2$ has a unique solution since its determinant is
exactly $W_{1,\pm}(z)$. The resolution of this system leads to the conclusion with the choice $d_2=1$.
\end{proof}

\begin{remark}
{\rm
The choice (\ref{geneigenfcts}) was guided by the simple case when $N=2$ and $V_k=0, k=1,2$.
In that case, we recover the standard generalized eigenfunctions, namely
 $$F_{z^2,1}^{\pm,1}(x)=e^{\pm izx}, \forall x>0,$$
as well as
 $$F_{z^2, 2}^{\pm,1}(x)=e^{\mp izx}, \forall x>0.$$
}
\end{remark}

According to   Lemma \ref{lpropJosti}, we see that
$$
c_{j,+,1}(z)=-\frac{iz N}{W_{j,+}(z)},
$$
which is always different from 0 if $z\in \CC^+, z\ne 0$, while
$$
c_{j,-,1}(z)=\frac{iz f_{j+1,+}(z,0)}{W_{j,-}(z)}
\sum_{k=1}^N
\frac{1-R_{k,2}(z)}{1+R_{k,2}(z)},
$$
is not clearly different from zero.
This is investigated in the next Lemma

\begin{lemma}\label{rgeneigenfcts}
Under the assumption (\ref{assumptionserge}), there exists $\kappa>0$ small enough such that
$$
s(z):=\sum_{k=1}^N
\frac{1-R_{k,2}(z)}{1+R_{k,2}(z)},
$$
satisfies
\beq\label{estsz}
|s(z)|\geq  C,\forall z\in B_\kappa,
\eeq
for some $C>0$.
\end{lemma}
\begin{proof}
Clearly $s$ is continuous on $B_\kappa\setminus\{0\}$ for $\kappa$ small enough, hence we first analyze the behaviour of $s$
near $z=0$.

For $z\in B_\kappa\setminus\{0\}$ and $k\in\{1,\ldots, N\},$ we write
$$
s_k(z):=\frac{1-R_{k,2}(z)}{1+R_{k,2}(z)}=\frac{1-R_{k,2}(z)}{T_{k}(z)}
\frac{T_{k}(z)}{1+R_{k,2}(z)}.
$$
The absolute value of the second factor is uniformly bounded from below
 on $B_\kappa$ thanks to Lemmas \ref{lpropJosti} and \ref{leigenfcts+nonnulle}.

For the first factor, we distinguish between the generic and the exceptional case:
In the exceptional case,
$$
|T_{k}(z)|\geq c_k, \forall z\in B_\kappa,
$$
for some $c_k>0$ (and $\kappa$ small enough) and therefore $s_k$ is continuous on $B_\kappa$.

In the generic case, using (\ref{expTj}) and (\ref{expR2j}),
we may write
$$
\frac{1-R_{k,2}(z)}{T_{k}(z)}=1-\int_{0}^{+\infty}
\frac{1+e^{2izy}}{2iz}  \tilde V_k(y) m_{k,+}(z,y) \, dy\quad \forall z\in \CC^+, z\ne 0.
$$
As underlined before, the derivative $\dot{m}_{k,+}$ of $m_{k,+}$
with respect to $k$ exists, is continuous on $\CC^+$and satisfies
(\ref{estmdot}). Accordingly, using the mean value theorem and the
dominated convergence theorem, we get for all $z\ne 0$ small enough
$$
\frac{1-R_{k,2}(z)}{T_k(z)}=1-\frac{\nu_k}{iz}+r_k(z),
$$
where $r_k$ is a continuous function at $z=0$
and  $\nu_k=\int_0^{+\infty}   V_k(t) m_{k,+}(0,t)\,dt$ (that is different from zero
because we are in the generic case).

In the same manner we can refine (\ref{serge10}) and prove that
$$
\frac{1+R_{k,2}(z)}{T_k(z)}=\gamma_k+z r^{(1)}_k(z),
$$
where $r^{(1)}$ is a continuous function at $z=0$ and
$\gamma_k=1+\int_0^{+\infty} t V_k(t) m_{k,+}(0,t)\,dt$ that   is a real number
 different from 0 by our
hypothesis (\ref{assumptionserge}).
Consequently for $z$ small enough we will get
\be\label{serge100}
\frac{T_k(z)}{1+R_{k,2}(z)}=\gamma_k^{-1}+z r^{(2)}_k(z),
\ee
where $r^{(2)}$ is a continuous function at $z=0$.

The two previous expansions show that for all $z\ne 0$ small enough
$$
s_k(z)=-\frac{\nu_k}{i \gamma_k z}+   r^{(3)}_k(z),
$$
where $r^{(3)}$ is a continuous function at $z=0$.

In summary, we have obtained that for all $z\ne 0$ small enough
$$
s(z)=-\frac{1}{i z} \sum_{k \rm generic}   \frac{\nu_k}{\gamma_k}+r(z),
$$
where $r$ is a continuous function at $z=0$.

Now we can distinguish two cases:
\\
i) If $\ds \sum_{k \rm generic}   \frac{\nu_k}{\gamma_k}=0$,
then $s$ is continuous at $z=0$, and therefore $s$ is continuous on $B_\kappa$.
\\
ii) If $K:= \ds \sum_{k \rm generic}   \frac{\nu_k }{\gamma_k}\ne 0$, then
$s$ blows up at $z=0$ and therefore
there exists $\delta_0$ small enough such that
\be\label{serge20}
|s(z)|\geq \frac{K}{2|z|}, \forall |z|<\delta_0.
\ee

Now for $|z|$ large, by
(\ref{bornitude1}) we have
$$
\lim_{|z|\to {+\infty}} s_k(z)=1,
$$
hence there exists $R_0$ large enough such that
\be\label{estszlarge}
\Re s(z)\geq \frac{N}{2},\forall z\in B_\kappa: |z|>R_0.
\ee

For small value of $|z|$, we  first restrict ourselves on the real line.
First we notice that
$$
\Re s_k(z)=\Re \frac{1-R_{k,2}(z)}{1+R_{k,2}(z)} =\frac{1-|R_{k,2}(z)|^2}{|1+R_{k,2}(z)|^2}.
$$
But
according to parts II and  V of Theorem 1 of \cite{deift-trub},
$$
|R_{k,2}(z)|<1,\quad \forall z\in \RR, z\ne 0,
$$
and therefore
$$
\Re s_k(z)>0,\quad \forall z\in \RR, z\ne 0.
$$
Now thanks to (\ref{serge10}) and to the relation
$$
1-|R_{k,2}(z)|^2=|T_k(z)|^2,$$
valid for all real numbers $z$,
we deduce that
$$
\lim_{z\to 0\atop z\in \RR}\frac{1-|R_{k,2}(z)|^2}{|1+R_{k,2}(z)|^2}=\frac{1}{\gamma_{k}^2},
$$
where $\gamma_k=1+\int_0^{+\infty} t V_j(t) m_{j,+}(0,t)\,dt$ that by hypothesis is a real number
 different from 0.

This shows that
$$
\lim_{z\to 0\atop z\in \RR} \Re s(z)= \sum_{k=1}^N\frac{1}{\gamma_{k}^2},
$$
and consequently as $\Re s$ is a continuous function on $\RR$ that is different from zero
for all real numbers,  due to (\ref{estszlarge}), it satisfies
\be\label{estszreel}
\Re s(z)\geq C,\forall z\in \RR,
\ee
for some $C>0$.

In the first case mentioned before, namely if $K=0$,
then  by the uniform continuity of $\Re s$ on the compact set $B_\kappa\cap\{z_1\in \CC: 0\leq z_1\leq R_0\}$, where $R_0$ is the parameter introduced above, we deduce that
\be\label{estszcompact}
\Re s(z)\geq C/2,\forall z\in B_{\kappa'}\cap \{z_1\in \CC: |z_1|\leq R_0\},
\ee
if $\kappa'$ is chosen small enough.
In that case the conclusion directly follows from (\ref{estszlarge}) and (\ref{estszcompact}).

In the case when $K\ne 0$, we use the uniform continuity of $\Re s$ on the compact set $B_\kappa\cap\{ z_1 \in \CC:  \frac{\delta_0}{2}\leq  |z_1|\leq R_0\}$ (where $R_0,\delta_0$ are the parameter introduced above),
and (\ref{estszreel}) to conclude that
\be\label{estszcompactKne0}
\Re s(z)\geq C/2,\forall z\in B_{\kappa'}\cap \{z_1\in \CC: \frac{\delta_0}{2}\leq  |z_1|\leq R_0\},
\ee
if $\kappa'$ is chosen small enough.

In this second case the conclusion follows from (\ref{serge20}), (\ref{estszlarge}) and (\ref{estszcompact}).
\end{proof}

\begin{corollary}\label{corogeneigenfcts}
Under the assumption (\ref{assumptionserge}), for $\kappa>0$ small enough
there exist two positive constants $c_1, c_2$ such that
\beq
\label{estc1}
|c_{j,-,1}(z) W_{j,-}(z)|\geq c_1 |z|,\quad \forall z\in B_\kappa,\\
|c_{j,-,2}(z)|\leq c_2 |s(z)|,\quad \forall z\in B_\kappa.
\eeq
\end{corollary}
\begin{proof}
As
$$
c_{j,-,1}(z)=\frac{iz f_{j+1,+}(z,0)}{W_{j,-}(z)}
s(z),
$$
by the previous Lemma and Lemma \ref{leigenfcts+nonnulle}, we deduce that
(\ref{estc1}) holds.

By its definition and   Lemma \ref{lpropJosti}, we may write
$$
c_{j,-,2}(z)=iz\frac{f_{j+1,+}(z,0)}{W_{j,-}(z)}
\left(1+\sum_{k\ne j}
\frac{1-R_{k}(z)}{1-R_{k}(z)}\right),
$$
hence thanks to the definition of $s(z)$, we obtain
$$
c_{j,-,2}(z)=iz\frac{f_{j+1,+}(z,0)}{W_{j,-}(z)}
\left(  \frac{2R_{j,2}(z)}{1+R_{j,2}(z)}  +s(z) \right).
$$
Now recalling that
$$
W_{j,-}(z)=-W_{j,+}(z)=-\frac{2iz}{T_j(z)},
$$
we can write
\be\label{serge200}
c_{j,-,2}(z)= -\frac{f_{j+1,+}(z,0)}{2}
\left(  \frac{2R_{j,2}(z)T_j(z)}{1+R_{j,2}(z)}  +s(z) T_j(z)\right).
\ee
By Lemmas \ref{lpropJosti}, \ref{leigenfcts+nonnulle}, \ref{lrj2borne} and \ref{rgeneigenfcts} we deduce that there exists $C_1>0$ such that
$$
|c_{j,-,2}(z)|\leq C_1
(1+|s(z)|)\leq (\frac{C_1}{C}+C_1) |s(z)|,
$$
with the constant $C$ from (\ref{estsz}).
\end{proof}

\begin{corollary}\label{corosdansH1}
Under the assumption (\ref{assumptionserge}),
and if $V_k\in L^1_\gamma(0,\infty)$ with $\gamma>5/2$, for all $k=1,\ldots, N$,
then  for all $R>0$, $s^{-1}$ belongs to $H^1(-R,R)$.
\end{corollary}
\begin{proof}
With the notation from the previous Lemma, we see that
$r_k$ is given by
$$
r_k(z)=\int_{0}^{+\infty}
\frac{V_k(y)}{2iz}  \left( 2m_{k,+}(0,y) -(1+e^{2izy})m_{k,+}(z,y)\right) \, dy,
$$
and  is continuous on $\RR$.
Moreover for $z\in \RR^*=\RR\setminus\{0\}$ we easily see that  $r_k$ is differentiable at $z$
and that
$$
\dot{r}_k(z)=
-\int_{0}^{+\infty}
\frac{V_k(y)}{2iz^2}  \left(2m_{k,+}(0,y) -g_{k,+}(z,y)
+z \dot{g}_{k,+}(z,y)
\right) \, dy.
$$
where for shortness we have set
$$
g_{k,+}(z,y):=(1+e^{2izy})m_{k,+}(z,y).
$$
But the mean value theorem implies that
$$
g_{k,+}(z,y)=2m_{k,+}(0,y)+z\dot{g}_{k,+}(\theta z,y),
$$
for some $\theta \in (0,1)$ and therefore
$$
\dot{r}_k(z)=
\int_{0}^{+\infty}
\frac{V_k(y)}{2iz}  \left(\dot{g}_{k,+}(\theta z,y)-\dot{g}_{k,+}(z,y)
\right) \, dy, \quad \forall z\in \RR^*.
$$
As
$$
\dot{g}_{k,+}(z,y)=2iy e^{2izy} m_{k,+}(z,y)+(1+e^{2izy})\dot{m}_{k,+}(z,y),
$$
the previous identity can be equivalently written
\beqs
\dot{r}_k(z)&=&
\int_{0}^{+\infty}
V_k(y)   \big(y \, m_{k,+}(\theta z,y)\frac{e^{2i\theta zy} -e^{2izy}}{z}
\\
&+& y e^{2izy} \frac{m_{k,+}(\theta z,y)-m_{k,+}(z,y)}{z}\\
&+& \frac{e^{2i\theta zy} -e^{2izy}}{2iz} \dot{m}_{k,+}(z,y)\\
&+& (1+e^{2i\theta zy}) \frac{\dot{m}_{k,+}(\theta z,y)-\dot{m}_{k,+}(z,y)}{2iz}
\big) \, dy, \quad \forall z\in \RR^*.
\eeqs
Again by the mean value theorem we get
\beqs
\dot{r}_k(z)&=&
\int_{0}^{+\infty}
V_k(y)   \big(2i y^2 m_{k,+}(\theta z,y) e^{2i\theta' zy} (\theta-1)
\\
&+& y e^{2izy}  \dot{m}_{k,+}(\theta'' z,y) (\theta-1)\\
&+&  y e^{2i\theta' zy} (\theta-1) \dot{m}_{k,+}(z,y)\\
&+& (1+e^{2i\theta zy}) \frac{\dot{m}_{k,+}(\theta z,y)-\dot{m}_{k,+}(z,y)}{2iz}
\big) \, dy, \quad \forall z\in \RR^*,
\eeqs
for some $\theta', \theta''\in (\theta, 1)$. Note that we cannot apply the mean value theorem to the last term since $\dot{m}_{k,+}$ is not differentiable. But according to Lemma 2.2 of \cite{weder} we have
\be\label{estl2.2weder}
|\dot{m}_{k,+}(z,y)-\dot{m}_{k,+}(0,y)|\leq C |z|^{\gamma-2}, \quad \forall y\geq 0,
\ee
for some $C>0$ independent of $z$ and $y$. This estimate, (\ref{estm+}) and  (\ref{estmdot}) lead to
$$
|\dot{r}_k(z)|\leq C
\int_{0}^{+\infty}
|V_k(y)|  (y^2+y+|z|^{\gamma-2}) \, dy, \quad \forall z\in \RR^*.
$$
for some $C>0$.
Hence according to our hypothesis on $V_k$, we get
$$
|\dot{r}_k(z)|\leq C_1 (1+|z|^{\gamma-3}), \quad \forall z\in \RR^*,
$$
for some $C_1>0$.

This estimate and the continuity of $r_k$ imply that $r_k$ belong to $H^1(-R,R)$ for any $R>0$
due to the hypothesis $\gamma>5/2$.

In the same way we need to precise the splitting (\ref{serge100}) on the real line (actually near 0).
For that purpose, we consider
$$
g_k(z):=\frac{m_{k,+}(z,0)-m_{k,+}(0,0)}{z}, \forall z\in \RR,
$$
and show that $g_k$ belongs to $H^1(-R,R)$ for any $R>0$.
First $g_k$ is continuous at $0$ because $m_{k,+}(z,0)$ is in $C^1(\RR)$.
Second by Leibniz's rule we have
$$
\dot{g}_k(z)=\frac{\dot{m}_{k,+}(z,0) z-(m_{k,+}(z,0)-m_{k,+}(0,0))}{z^2}
$$
and therefore by  the mean value theorem we get
$$
\dot{g}_k(z)=\frac{\dot{m}_{k,+}(z,0) -\dot{m}_{k,+}(\theta z,0)}{z},
$$
for some $\theta\in (0,1)$ and we conclude by
 (\ref{estl2.2weder}).

But
we see that
$$
\frac{(m_{k,+}(z,0))^{-1}}{z}=\frac{(m_{k,+}(0,0))^{-1}}{z}+h_k(z)=\frac{1}{\gamma_k z}+h_k(z)
$$
with
$$
h_k(z)=\frac{m_{k,+}(z,0)-m_{k,+}(0,0)}{z m_{k,+}(z,0) m_{k,+}(0,0)}=\frac{g_k(z)}{m_{k,+}(z,0) m_{k,+}(0,0)}.
$$
According to the previous considerations, $g_k$ belongs to $H^1(-R,R)$, for any $R>0$
and since $m_{k,+}(\cdot,0)$ belongs to $C^1(\RR)$ and is uniformly bounded from below (due to Lemmas \ref{lpropJosti} and \ref{leigenfcts+nonnulle}), $\frac{1}{m_{k,+}(\cdot,0)}$ is also in $C^1(\RR)$.
Therefore $h_k$ also belongs to $H^1(-R,R)$, for any $R>0$.

Coming back to $s$,  recalling that
$$s(z)=\sum_{k=1}^N
 (1+\frac{i\nu_k}{z}+r_k(z)) (m_{k,+}(z,0))^{-1},
$$
we have finally shown that
$$
s(z)=i\frac{K}{z}+r_s(z),
$$
where $r_s$  belongs to $H^1(-R,R)$, for any $R>0$.

Now we distinguish the case $K=0$ to the other one:
In the first case, we have that $s=r_s$ belongs to $H^1(-R,R)$, for any $R>0$
and since $s$ is uniformly bounded from below by the previous Lemma, we deduce that
$\frac{1}{s}$ belongs to $H^1(-R,R)$, for any $R>0$.

If $K\ne 0$, then
$$
\frac{1}{s(z)}=\frac{z}{iK+zr_s(z)},
$$
that is a continuous function in $\RR$
and moreover for $z\in \RR^*$, we have after elementary calculations
$$
\frac{d}{dz} \frac{1}{s}(z)=\frac{iK-z^2\dot{r}_s(z)}{(iK+zr_s(z))^2}.
$$
Since this right-hand side is in $L^2(-R,R)$, for any $R>0$ (because the denominator is different from zero near $z=0$, while by   the previous Lemma, for any $z\in \RR^*$
$s(z)\geq C$ is equivalent to $|iK+zr_s(z)|\geq C|z|$), we still conclude that
$\frac{1}{s}$ belongs to $H^1(-R,R)$, for any $R>0$.
\end{proof}

\begin{corollary}\label{corouncoefdansH1}
Under the assumption (\ref{assumptionserge}),
and if $V_k\in L^1_\gamma(0,+\infty)$ with $\gamma>5/2$,
then   the function
$$
{\RR} \to {\CC}: z \to \frac{c_{j,-,2}(z) }{f_{j+1,+}(z,0) s(z)},
$$
belongs to $H^1(-R,R)$  for all $R>0$.
\end{corollary}
\begin{proof}
By (\ref{serge200}), we see that
$$
\frac{c_{j,-,2}(z) }{f_{j+1,+}(z,0) s(z)}=
 -\frac{1}{2}
\left(  \frac{2  R_{j,2}(z)T_j(z)}{(1+R_{j,2}(z)) s(z)}  + T_j(z)\right)
=-\frac{1}{2 }
\left(  \frac{2\ R_{j,2}(z)}{f_{j,+}(z,0) s(z)}  + T_j(z)\right).
$$
But according to Remark 10 of \cite{deift-trub}, $T_j$ is analytic in a neighbourhood of the real line,
hence it is at least in $C^1(\RR)$. On the other hand $f_{j,+}(z,0)=m_{j,+}(z,0)$
is $C^1(\RR)$ due to Remark 3 of \cite{deift-trub}, hence $\frac{1}{f_{j,+}(z,0)}$ has the same property due to Lemma \ref{leigenfcts+nonnulle}. Finally the identity (\ref{propJosti+0})
of    Lemma \ref{lpropJosti} yields
$$
R_{j,2}(z)=f_{j,+} (z,0) T_j(z)-1,
$$
hence it also belongs to $C^1(\RR)$.

The conclusion follows from the previous Corollary and these regularity properties
(the product of a $C^1$ function with a $H^1$ function is still in $H^1$).
\end{proof}

\begin{Definition}[Kernel of the resolvent] \label{def.K}
Let   the assumption (\ref{assumptionserge}) be satisfied, then
for all $z\in B_\kappa, z\ne 0$, all $j \in \{ 1, \cdots, N\}$,
and all $x\in R_j$,
we define (modulo $N$)
\[ K(x,x',z^2) = \left\{ \begin{array}{ll}
        \frac{1}{W_j(z)} F^{-,j}_{z^2,j}(x)
              F^{-,j+1}_{z^2,j}(x'), & \hbox{ for } x' \in  {R_j},
              \, x'> x,
\\
        \frac{1}{W_j(z)} F^{-,j+1}_{z^2,j}(x)
              F^{-,j}_{z^2,j}(x'), & \hbox{ for } x'\in R_j, \, x'< x,
\\
        \frac{1}{W_j(z)} F^{-,j+1}_{z^2,j}(x)
              F^{-,j}_{z^2,k}(x'), & \hbox{ for } x' \in  {R_k}, k
              \neq j,
        \end{array} \right.
\]
where $W_j(z)= c_{j,-,1}(z) d_{j+1,j,-}(z) W_{j,-}(z)$.
\end{Definition}

\begin{theorem} \label{theo1}
Let   the assumption (\ref{assumptionserge}) be satisfied and let $f \in \cal H$. Then, for $x \in {\cal R}$ and $z\in B_\kappa$ such that
$\Im z> 0$, we have
\be\label{resolvantformula} [R(z^2, H)f](x)= \int_{{\cal R}} K(x, x', z^2) f(x') \; dx'.
\ee
\end{theorem}
%
\begin{proof}
Fix $j\in \{1,\ldots, N\}$, and $z$ as in the statement. Then  we notice that
the Wronskian $W_{j}(z)$ between  $F^{-,j}_{z^2,j}$ and  $F^{-,j+1}_{z^2,j}$ is different from zero,
namely by Lemma \ref{lgeneigenfcts}  we have
\beqs
W_{j}(z)&=&[F^{-,j}_{z^2,j}, F^{-,j+1}_{z^2,j}](x)\\
&=& F^{-,j}_{z^2,j}(x) \left(F^{-,j+1}_{z^2,j}\right)'(x)-\left(F^{-,j}_{z^2,j}\right)'(x) F^{-,j+1}_{z^2,j}(x)
\\
&=&(c_{j,-,1}(z) f'_{j,-}(z,x)+c_{j,-,2}(z) f'_{j,+}(z,x)) d_{j+1,j,-}(z)  f_{j,+}(z,x)
\\
&-&(c_{j,-,1}(z) f_{j,-}(z,x)+c_{j,-,2}(z) f_{j,+}(z,x)) d_{j+1,j,-}(z)  f'_{j,+}(z,x)
\\
&=&c_{j,-,1}(z) d_{j+1,j,-}(z) W_{j,-}(z).
\eeqs
Hence by Lemma  \ref{leigenfcts+nonnulle} and Corollary  \ref{corogeneigenfcts}
this Wronskian is different from zero.

Consequently
the same arguments than in Proposition 3.2 of \cite{meh} show that (\ref{resolvantformula}) holds.
The main ingredient is that we can apply the dominated convergence theorem because
the generalized eigenfunction
$F^{-,j}_{z^2,k}$ is in $L^2(R_k)$ if $j\ne k$.
\end{proof}

\begin{remark}
{\rm
The choice of the kernel comes from this Theorem because
$F^{+,j}_{z^2,k}$ is not in $L^2(R_k)$ if $j\ne k$.
}
\end{remark}

 Here and below the complex square root is chosen in such a way
that $\sqrt{r \cdot e^{i \phi}} = \sqrt{r} e^{i \phi/2}$ with $r>0$ and $\phi
\in [-\pi,\pi)$. Accordingly for any positive real number $\lambda$ and any $\eps>0$, we will define
$$
z_\eps=\sqrt{\lambda+i \eps}
$$
that will be in $\CC^+$.

\begin{theorem}[Limiting absorption principle] \label{lim.abs}
Let   the assumption (\ref{assumptionserge}) be satisfied.
Let $\delta > 0$ be fixed. Then for all  real numbers $\lambda>0$, $0 < \eps < \delta$ and
$(x,x') \in {\cal R}^2$ we have
\begin{enumerate}
\item $\lim_{\alpha \rightarrow 0\atop \alpha>0} K(x,x',z_\alpha^2) =
      K(x,x',\lambda)$,
\item $|K(x,x',z_\eps^2)| \leq \frac{C}{\sqrt{\lambda}}
 e^{\gamma  (x + x')}$,
    where $0<\gamma<\max\{1,\delta\}$.
\end{enumerate}
\end{theorem}
%
\begin{proof}
The first part of the Theorem is direct since $\lambda + i \alpha$ tends to $\lambda$
as $\alpha>0$ tends to 0 and consequently
$$
\sqrt{\lambda + i \alpha}\to \sqrt{\lambda},
$$
as $\alpha>0$ tends to 0. We further use the fact that  the functions $f_{j,\pm}(\cdot, x)$
and $f'_{j,\pm}(\cdot, x)$ are continuous in $\CC^+$ for any fixed $x\in \RR$.

For the second  part of the Theorem, we first use the estimates (\ref{estf+}) and
 (\ref{estf-}), this last one implying
\be\label{estf-eps}
|f_{j,-}(z_\eps,x)|\leq  C (1+x) e^{\Im z_\eps x}\leq
C (1+x) e^{\max\{1,\delta\} x}, \quad \forall x\in [0,{+\infty}),
\ee
where we have used the property
$$
\Im z_\eps=|\Im \sqrt{\lambda+i \eps}|\leq \max\{1, \Im (\lambda+i \eps)\}=
\max\{1, \eps\}.$$

Notice that by the definition $W_{j}(z)=c_{j,-,1}(z) d_{j+1,j,-}(z) W_{j,-}(z)$
and by Lemma  \ref{leigenfcts+nonnulle} and Corollary  \ref{corogeneigenfcts}, we get
\be\label{upperwj}
|W_{j}(z)|\geq C |z|,
\ee
for some $C>0$.

Now we distinguish between the following three cases:\\
1. If $x, x'\in R_j$ with $x'>x$, then
\beqs
K(x,x',z_\eps^2)&=&\frac{1}{W_j(z_\eps)} F^{-,j}_{z_\eps^2,j}(x)
              F^{-,j+1}_{z_\eps^2,j}(x') \\
&=&
\frac{1}{W_j(z_\eps)} \Big(c_{j,-,1}(z_\eps) f_{j,-}(z_\eps,x)+c_{j,-,2}(z_\eps) f_{j,+}(z_\eps,x))\Big) d_{j+1,j,-}(z_\eps)  f_{j,+}(z_\eps,x')
\\
&=&\frac{1}{W_{j,-}(z_\eps)}  f_{j,-}(z_\eps,x)f_{j,+}(z_\eps,x')
+\frac{c_{j,-,2}(z_\eps) }{iz_\eps  f_{j+1,+}(z_\eps,0) s(z_\eps)}
 f_{j,+}(z_\eps,x)  f_{j,+}(z_\eps,x').
\eeqs
As there exists $c>0$ such that
$$
|W_{j,-}(z)|\geq c |z|, \forall z\in \CC^+,
$$
by Lemma  \ref{leigenfcts+nonnulle} and Corollary  \ref{corogeneigenfcts}, we obtain
$$
|K(x,x',z_\eps)|\leq \frac{C}{|z_\eps|}(|f_{j,-}(z_\eps,x)|+|f_{j,+}(z_\eps,x)|) |f_{j,+}(z_\eps,x')|.
$$
The estimates (\ref{estf+}) and
 (\ref{estf-eps}) then yields
\be\label{serge40}
|K(x,x',z_\eps^2)|\leq \frac{C}{|z_\eps|}(1+ (1+x) e^{\max\{1,\delta\} x}).
\ee
2. If $x, x'\in R_j$ with $x'>x$, then
$$
K(x,x',z_\eps^2)=\frac{1}{W_j(z_\eps)} F^{-,j+1}_{z_\eps^2,j}(x)
              F^{-,j}_{z_\eps^2,j}(x'),
$$
and the above arguments (by simply exchanging the role of $x$ and $x'$) yields
\be\label{serge50}
|K(x,x',z_\eps)|\leq \frac{C}{|z_\eps|}(1+ (1+x') e^{\max\{1,\delta\} x'}).
\ee
3. If $x\in R_j$ and $x'\in R_k$ with $k\ne j$, we have
\beqs
K(x,x',z_\eps^2)&=&\frac{1}{W_j(z_\eps)} F^{-,j+1}_{z_\eps^2,j}(x)
              F^{-,j}_{z_\eps^2,k}(x')
\\
&=&\frac{1}{W_j(z_\eps)}
d_{j+1,j,-}(z_\eps)  f_{j,+}(z_\eps,x) d_{j+1,k,-}(z_\eps)  f_{k,+}(z_\eps,x).
\eeqs
Hence by Lemma  \ref{leigenfcts+nonnulle} and the estimates (\ref{estf+}) and
 (\ref{upperwj}), we obtain
\be\label{serge60}
|K(x,x',z_\eps^2)|\leq \frac{C}{|z_\eps|}.
\ee

The estimates (\ref{serge40}), (\ref{serge50}) and (\ref{serge60}) imply the conclusion
since
$|z_\eps|>\sqrt{\lambda}$.
\end{proof}

\begin{theorem} \label{res.id}
Take $f \in {\cal H} $ with a compact support and let $0\leq a < b < + \infty$. Then for any continuous
scalar function $h$ defined on the real line and for all $x \in R_j$, we have
$$( h(H) E(a,b)f)(x) =  -
\frac{1}{\pi}
 \int_{(a,b) } h(\lambda)
          \sum\limits_{k=1}^{N}
          \int\limits_{R_k} f(x')
          \Im K(x,x',\lambda) \; dx'
           \; d\lambda,
$$
where $E$ is the resolution of the identity of $H$.
\end{theorem}
\begin{proof}
The proof is similar to the one  of Lemma~3.13 of \cite{fam2}
(see also Proposition 4.5 of
\cite{FAMetall10}) and is therefore omitted.
The main ingredients are the use of   Stone's formula,
Theorem~\ref{theo1} and    the limiting absorption principle
Theorem \ref{lim.abs} (that allows to apply the dominated convergence theorem).
\end{proof}
%
\begin{remark} \label{type of spectrum}
 Theorem \ref{res.id} directly implies that
$$\sigma(H E[0,+\infty)) = \sigma_{ac}(H E[0,+\infty)) = [0,+\infty)
\hbox{ and } \sigma_{pp}(H) \subset (-\infty,0),$$
where $\sigma_{ac}$ is the absolutely continuous spectrum   and
$\sigma_{pp}$ the pure point spectrum.  The additional informations
that
$$\sigma(H E(-\infty,0)) = \sigma_{pp}(H) $$
and that this set is finite follow from chapter 2 of \cite{berzin}.
\end{remark}
%
\section{Proof of Theorems \ref{mainresult} and \ref{high frequency perturbation}} \label{perturbed}

       The proof of the $L^\infty-$time decay will be carried out by
manipulating the solution formula in a way to reduce the problem to
the well known case of the free Schr\"{o}dinger equation on the line
\cite{ReedSimonII}, p.~60.

We shall decompose an general initial conditions into a part with a
spectral representation with compact support and a part with a
sufficiently high lower cutoff energy (frequency). The technique
will be different in the two cases.

\subsection{High energy limit}

For high energy (frequency) initial conditions, we can use an
expansion (called Born series) of the resolvent of the Hamiltonian
with potential in terms of the free resolvent (Proposition \ref{Born
series}). To this end we use a formula for the free resolvent
established in \cite{meh}. This leads to a corresponding  expansion
of the Schr\"{o}dinger group via Stone's formula. Then we adapt a
technique of \cite{GS} to extract the expression corresponding to
the Schr\"{o}dinger group on the line to the formulas of the
transmission problem, see Theorem \ref{high frequency perturbation
section}, part 1. While doing this, we improve the calculations of
\cite{GS} in the sense that we find an explicit expression for the
coefficient of the time decay in terms of the cutoff frequency and
the potential. This explicit knowledge is essential to deduce from
this the perturbation theorem \ref{high frequency perturbation
section}, part 3, using the fact that the free Schr\"{o}dinger group
is the first term of the expansion. The results of this section are
of independent interest and Theorem \ref{high frequency perturbation
section}, part 3 seems to be new even on the line.

\begin{proposition}\label{Born series}
Let
$R_0(\lambda + i \varepsilon) = \left( - \frac{d^2}{dx^2} - (\lambda
+ i \varepsilon) \right)^{-1}$
and
$R_V(\lambda + i \varepsilon) = \left( H - (\lambda + i \varepsilon)
\right)^{-1}.$ Then we have
\begin{enumerate}
  \item
the representation
  $$\lim_{\varepsilon \rightarrow 0, \varepsilon > 0}
  [R_0(\lambda + i\varepsilon)f](x)
  = [R_0(\lambda + i0)f](x) = \int_{\cal R} K_0(x,x',\lambda + i0)
  f(x') dx'$$
\\
  for almost all $x \in \RR$  and $f \in L^2({\cal R})$
  with
\begin{equation}\label{free kernel}
K_0(x,x^\prime,\lambda \pm i 0) = \frac{\mp i}{N \sqrt{\lambda}} \,
\left\{
\begin{array}{ll}
(1 - \frac{N}{2})e^{\pm i (x + x') \sqrt{\lambda}}
+
\frac{N}{2} \, e^{\pm i  |x - x'| \sqrt{\lambda}} , \, x' \in
\overline{R_j}, \\
(1 - \frac{N}{2})e^{\pm i (x + x') \sqrt{\lambda}}
+
\frac{N}{2} \, e^{\pm i  (x - x') \sqrt{\lambda}} , \, x' \in
\overline{R_k}, k \neq j \, ,
\end{array}
\right.
\end{equation}
  \item
the estimate
\be
\label{estkernelho} \left| K_0(x,x^\prime,\lambda \pm
0) \right| \leq \frac{N-1}{N \sqrt{\lambda}}, \, \forall \,
(x,x^\prime) \in {\cal R}^2,
\ee
  \item
the following expansion: suppose $N \geq 2$,
let $0<q_*<1$ and
$\lambda > \lambda_* = \frac{4(N-1)^2 \|V\|_1^2}{ N^2 q_*^2}$.
Then
$$
\left\langle R_V (\lambda \pm i 0) f,g \right\rangle = \ds \sum_{k
\geq 0 } \left\langle R_0 (\lambda \pm i0) (-V R_0(\lambda \pm i
0))^k f,g \right\rangle
$$
for any $V,f,g \in L^1({\cal R})$. The $+ (-)$ sign is valid, if
$\Im \lambda > 0$ (respectively $\Im \lambda < 0$).
\end{enumerate}
\end{proposition}
\begin{proof}~
\\
\underline{1.:}
\\
Direct consequence of \cite{meh}.
\\
\underline{2.:}
\\
Follows from 1.
\\
\underline{3.:}
\\
From 2. and the assumption on $V$ it follows
$$
\|V R_0(\lambda \pm i 0)f\|_{1}
\leq \frac{N-1}{N \sqrt{\lambda}} \, \|V \|_1 \|f \|_1.
$$ Due to
(\ref{asymp.osc}) we see, that the Jost functions are bounded for
fixed $\lambda.$ Therefore one has
\[
    R_V(\lambda - i 0) g \in L^\infty ({\cal R}) \hbox{ for }
    \lambda > 0.
\]
Hence
\beqs
\left| \left\langle R_V (\lambda + i 0) (V R_0(\lambda + i 0)^k f,g
\right\rangle \right|
&\leq&
\| ( V R_0(\lambda + i 0))^k f \|_1 \,
\|R_V (\lambda - i 0) g \|_\infty
\\
&\leq&
\left(\frac{N-1}{N \sqrt{\lambda}}\right)^{k} \left\| V \right\|_1^k
\, \left\|f \right\|_1 \, \left\| R_V (\lambda - i
0)g\right\|_\infty
\\
&=&
q(\lambda)^k \, \left\|f \right\|_1 \,
\left\| R_V (\lambda - i 0)g\right\|_\infty
\eeqs
with $q(\lambda):=\frac{N-1}{N \sqrt{\lambda}}.$ Our assumption
$\lambda_* < \lambda$ implies
\[
q(\lambda)  <
\frac{4(N-1)}{N \sqrt{\lambda_*}} \, \|V \|_1
= q_* < 1.
\]
Therefore the series from the statement of 3. converges. The
equality comes from simple calculations.
\end{proof}
Note that the factor $4$ in the definition of $\lambda_*$ is not
necessary in this Proposition, but will be necessary later on.

Now we shall estimate the $L^1-$norm of the Fourier transform of a
frequency band cutoff function times all negative powers
$\lambda^{-n}$ of the frequency. These quantities measure the
influence of the cutoff function on the terms of the expansion of
the high frequency part of the solution. Note that in \cite{GS} it
is claimed (only indicating the steps of a proof), that there exists
a bound which is independent of $n.$ This does not seem to be
rigorously correct: writing down the details of the proof sketched
in \cite{GS}, we find an explicit bound in terms of certain norms of
the cutoff function, but which grows linearly in $n$. But this
growth has no influence on the convergence of the expansion of the
solution. Nevertheless the explicitness of the estimate will allow
us to give an upper bound of the coefficient of the time decay of
the solution.

\begin{Definition}
Let $\phi \in C^\infty(\RR)$ be such that
$0 \leq \phi(\lambda) \leq 1$
and $\phi(\lambda) = 1$ if $|\lambda| \leq 1$ and $\phi(\lambda) =0$
if $|\lambda| \geq 2.$
Let $\lambda_0 \geq 1$ and $L > 2 \lambda_0.$ Define
\begin{enumerate}
  \item
  $\cO \in C^\infty([1,\infty[)$ by
  $\cO (\lambda):= 1 - \phi(\frac{\lambda}{\lambda_0}), \, \lambda \geq 1,$
  \item
  $\cOL \in C^\infty([1,\infty[)$ by
  $\cOL (\lambda)
  := \cO (\lambda)\phi(\frac{\lambda}{\lambda_0}), \, \lambda \geq 1.$
\end{enumerate}
\end{Definition}


\begin{theorem} \label{p1}
For $n \in \nline, \lambda_0 \geq 1, L > 2 \lambda_0$ it holds\footnote{We write shortly $f^{\vee}={\mathcal F}^{-1} f$}:
$$
\left\|\left[\cOL (\lambda^2) \lambda^{-n}\right]^{\vee}\right\|_1
\leq c(n) \, \lambda_0^{-n/2},
$$
with  $c(0) = N_1 + N_1^2, \, c(1) = 2 \, (N_1 +N_1^2) + 32 \sqrt{2}
\, N_2, \, c(n) = \frac{4}{n-1} + 32 \sqrt{2} N_2 n, \, n \geq 2$
and hence $c(n) \leq M n, \, n \geq 1$, where $N_1 =
\left\|\left[\phi(\lambda^2)\right]^{\vee}\right\|_1$, $N_2 =
\left\|\phi \right\|_{C^2(\rline)}$
 and $M = 32 \sqrt{2} \, \max\left\{N_1 + N_1^2,N_2 \right\}.$
\end{theorem}
\begin{proof}
The proof follows from Theorem \ref{thm5.6Kais} and Propositions
\ref{prop5.7Kais} and \ref{prop5.9Kais} below.
\end{proof}
\begin{proposition}\label{prop5.3Kais}
Suppose $\lambda_0, L \geq 1$ and $2 \lambda_0 < L$. Then we have
for all $\lambda \in \rline$ 
\begin{eqnarray*}
  \left| \cOL (\lambda^2) \right| \leq
\mathbbm{1}_{\left\{\sqrt{\lambda_0} \leq |\lambda| \leq
\sqrt{2L}\right\}} (\lambda),
  \\
 \left|\frac{d}{d\lambda} \left( \cOL(\lambda^2)\right)\right| \leq 2|\lambda|  \left\|\phi\right\|_{C^1(\rline)}\big(
\frac{1}{\lambda_0} \,  \mathbbm{1}_{\left\{\sqrt{\lambda_0} \leq
|\lambda| \leq \sqrt{2\lambda_0}\right\}} (\lambda) + \frac{1}{L} \,
\mathbbm{1}_{\left\{\sqrt{L} \leq |\lambda| \leq \sqrt{2L} \right\}}
(\lambda)\big),
  \\
 \left| \frac{d^2}{d\lambda^2} \left(\cOL (\lambda^2) \right)\right| \leq
\left\|\phi\right\|_{C^2(\rline)} \,\Big( \left( \frac{2}{\lambda_0}
+ \frac{4 |\lambda|^2}{\lambda^2_0} \right) \,
\mathbbm{1}_{\left\{\sqrt{\lambda_0} \leq |\lambda|\leq
\sqrt{2\lambda_0} \right\}} (\lambda)
\\
+\left(\frac{2}{L} + \frac{4|\lambda|^2}{L^2} \right) \,
\mathbbm{1}_{\left\{\sqrt{L} \leq |\lambda| \leq \sqrt{2L}
\right\}}\Big).
\end{eqnarray*}
\end{proposition}
\begin{proof}
Clearly we have
$$
\frac{d}{d\lambda}
\left(\phi\left(\frac{\lambda^2}{\alpha}\right)\right) =
\frac{2\lambda}{\alpha} \, \phi^\prime
\left(\frac{\lambda^2}{\alpha}\right) \hbox{ and }
\frac{d^2}{d\lambda^2}
\left(\phi\left(\frac{\lambda^2}{\alpha}\right)\right) =
\frac{2\lambda}{\alpha} \, \phi^\prime
\left(\frac{\lambda^2}{\alpha}\right) + \frac{4\lambda^2}{\alpha^2}
\, \phi^{\prime \prime} \left(\frac{\lambda^2}{\alpha}\right).
$$
Further we have for $\lambda_0, L$ and $\lambda \geq 1$
$$
\frac{d}{d\lambda} \left( \cOL (\lambda^2)\right) = - \left[ \phi
\left(\frac{\lambda^2}{\lambda_0}\right)\right]^\prime \, \phi
\left(\frac{\lambda^2}{L} \right) + \left(1 -
\phi\left(\frac{\lambda^2}{\lambda_0}\right)\right) \, \left[ \phi
\left(\frac{\lambda^2}{L}\right)\right]^\prime
$$
and
\begin{eqnarray*}
\frac{d^2}{d\lambda^2} \left(\cOL (\lambda^2)\right) = - \left[ \phi
\left(\frac{\lambda^2}{\lambda_0}\right)\right]^{\prime \prime} \,
\phi \left(\frac{\lambda^2}{L}\right) - 2 \left[\phi
\left(\frac{\lambda^2}{\lambda_0}\right)\right]^\prime \, \left[\phi
\left(\frac{\lambda^2}{L}\right)\right]^\prime
\\
+\left(1 - \phi\left(\frac{\lambda^2}{\lambda_0}\right)\right) \,
\left[ \phi \left(\frac{\lambda^2}{L}\right)\right]^{\prime \prime}.
\end{eqnarray*}
We estimate for $\alpha \geq 1$ and $\lambda \in \rline$:
$$
\left|\phi \left(\frac{\lambda^2}{\alpha}\right) \right| \leq
\left\| \phi \right\|_{C^0 (\rline)} \, \mathbbm{1}_{]- \infty,2]}
\left(\frac{\lambda^2}{\alpha}\right) = \left\| \phi \right\|_{C^0
(\rline)} \,\mathbbm{1}_{\left\{|\lambda| \leq \sqrt{2 \alpha}
\right\}} (\lambda)
$$
due to $\frac{\lambda^2}{\alpha} \leq 2 \Leftrightarrow |\lambda|
\leq \sqrt{2 \alpha}.$ Similarly we have
$$
\left|1 - \phi\left(\frac{\lambda^2}{\alpha}\right)\right| \leq
\mathbbm{1}_{[1,+\infty[} \left(\frac{\lambda^2}{\alpha}\right) =
\mathbbm{1}_{\left\{\sqrt{\alpha} \leq |\lambda| \right\}}
(\lambda).
$$
Further
$$
\left| \frac{d}{d\lambda} \left( \phi
\left(\frac{\lambda^2}{\alpha}\right)\right) \right| = \left|
\frac{2\lambda}{\alpha} \, \phi^\prime
\left(\frac{\lambda^2}{\alpha}\right)\right| \leq \frac{2
|\lambda|}{\alpha} \, \left\|\phi\right\|_{C^1(\rline)} \,
\mathbbm{1}_{\left\{\sqrt{\alpha} \leq |\lambda| \leq \sqrt{2
\alpha} \right\}} (\lambda)
$$
and
$$
\left| \frac{d^2}{d\lambda^2} \left(\phi
\left(\frac{\lambda^2}{\alpha}\right)\right)\right| \leq \left(
\frac{2}{\alpha} + \frac{4|\lambda|^2}{\alpha^2} \right) \,
\left\|\phi\right\|_{C^2(\rline)} \,
\mathbbm{1}_{\left\{\sqrt{\alpha} \leq |\lambda|\leq \sqrt{2 \alpha}
\right\}}.
$$
The three stated estimates directly follow from the previous
properties.
\end{proof}

\begin{proposition}\label{prop5.4Kais}
Let $n \in \nline^*$ and let $\lambda_0, L \geq 1$ with $2 \lambda_0
\leq L.$ Then, recalling that $N_2 = \left\|
\phi\right\|_{C^2(\rline)}$,
$$
\left\| \left[\cOL (\lambda^2) \, \lambda^{-n} \right]^{\vee} (\tau)
\, \tau^2 \right\|_\infty  \leq 16 \, \sqrt{2} \, N_2 \,
\lambda_0^{\frac{-n-1}{2}} \, n.
$$
\end{proposition}
\begin{proof}
By standard properties of the Fourier transform we have
$$
\left\| \left[\cOL (\lambda^2) \, \lambda^{-n} \right]^{\vee} (\tau)
\, \tau^2 \right\|_\infty = \left\| \left[(\cOL (\lambda^2) \,
\lambda^{-n})'' \right]^{\vee}  \right\|_\infty \leq \left\|(\cOL
(\lambda^2) \, \lambda^{-n})'' \right\|_1.
$$
Hence by Leibniz's rule and the previous proposition, we find that
\begin{eqnarray*}
\left\| \left[\cOL (\lambda^2) \, \lambda^{-n} \right]^{\vee} (\tau)
\, \tau^2 \right\|_\infty  \leq \int_{-\infty}^{+\infty}
\left|\frac{d^2}{d\lambda^2} \left( \cOL (\lambda^2) \right)\right|
\, |\lambda|^{-n} \, d\lambda
\\
+ 2 \, \int_{-\infty}^{+\infty} \left| \frac{d}{d\lambda} \cOL
(\lambda^2) \right| \, n \, |\lambda|^{-n-1} \, d \lambda +
\int_{-\infty}^{+\infty} \left| \cOL (\lambda^2) \right| \, n(n+1)
\, |\lambda|^{-n-2} \, d\lambda
\\
\leq \int_{-\infty}^{+\infty} \left[ \left( \frac{2}{\lambda_0} +
\frac{4 |\lambda|^2}{\lambda_0^2} \right) \, \left\|\phi
\right\|_{C^2(\rline)} \, \mathbbm{1}_{\left\{\sqrt{\lambda_0} \leq
|\lambda| \leq \sqrt{2 \lambda_0}  \right\}} (\lambda)
 \right.
\\
+\left. \left( \frac{2}{L} + \frac{4 |\lambda|^2}{L^2} \right) \,
\left\|\phi\right\|_{C^2(\rline)} \, \mathbbm{1}_{\left\{\sqrt{L}
\leq \sqrt{\lambda} \leq \sqrt{2L} \right\}} (\lambda)\right] \,
|\lambda|^{-n} \, d \lambda
\\
+ 2 \, \int_{-\infty}^{+ \infty} \left[ \frac{2|\lambda|}{\lambda_0}
\, \left\|\phi\right\|_{C^1(\rline)} \,
\mathbbm{1}_{\left\{\sqrt{\lambda_0} \leq |\lambda| \leq
\sqrt{2\lambda_0} \right\}} (\lambda) + \frac{2|\lambda|}{L} \,
\left\|\phi \right\|_{C^1(\rline)} \, \mathbbm{1}_{\left\{\sqrt{L}
\leq |\lambda| \leq \sqrt{2L} \right\}} (\lambda) \right] \, n \,
|\lambda|^{-n-1} \, d \lambda
\\+
\int_{-\infty}^{+\infty} \mathbbm{1}_{\left\{\sqrt{\lambda_0} \leq
|\lambda| \leq   \sqrt{2L} \right\}} (\lambda) \, n (n+1) \,
|\lambda|^{-n-2} \, d\lambda.
\end{eqnarray*}
By using that for $\alpha > 0$ we have
$$
|\lambda|^k \, \mathbbm{1}_{\left\{\sqrt{\alpha} \leq |\lambda| \leq
\sqrt{2\alpha} \right\}} (\lambda) \leq (2\alpha)^{k/2},
$$
we obtain
\begin{eqnarray*}
\left\| \left[\cOL (\lambda^2) \, \lambda^{-n} \right]^{\vee} (\tau)
\, \tau^2 \right\|_\infty   \leq   \left\|\phi\right\|_{C^2(\rline)}
\Big(\frac{10}{\lambda_0} \int_{\sqrt{\lambda_0} \leq |\lambda| \leq
\sqrt{2\lambda_0}} |\lambda|^{-n} \, d\lambda +
\frac{10}{L}\int_{\sqrt{L} \leq |\lambda| \leq \sqrt{2L}}
|\lambda|^{-n} \, d\lambda \Big)
\\
+ 2 \left\|\phi\right\|_{C^1(\rline)}
\,\Big(\frac{2\sqrt{2}}{\lambda_0^{1/2}}\,
 \int_{\sqrt{\lambda_0} \leq |\lambda| \leq \sqrt{2 \lambda_0}} |\lambda|^{-n-1} \,
d\lambda  +
 \frac{2\sqrt{2}}{L^{1/2}} \int_{\sqrt{L} \leq |\lambda| \leq \sqrt{2 L}} |\lambda|^{-n-1} \,
d\lambda \Big)
\\
+ n(n+1)\,\int_{\sqrt{\lambda_0} \leq |\lambda| \leq \sqrt{2L}}
|\lambda|^{-n-2} \, d\lambda .
\end{eqnarray*}
Calculating these integrals we find
\begin{eqnarray*}
\left\| \left[\cOL (\lambda^2) \, \lambda^{-n} \right]^{\vee} (\tau)
\, \tau^2 \right\|_\infty   \leq
 N_2
\Biggl(\frac{20}{\lambda_0(n-1)} \, \left(
\lambda_0^{\frac{-n+1}{2}} - (2 \lambda_0)^{\frac{-n+1}{2}} \right)
+ \frac{20}{L(n-1)} \, \left( L^{\frac{-n+1}{2}} -
(2L)^{\frac{-n+1}{2}} \right)
\\
+ \frac{8\sqrt{2}}{n\sqrt{\lambda_0}} \, \left(\lambda_0^{-n/2} -
(2\lambda_0)^{\frac{-n+1}{2}} \right) \, + \frac{8\sqrt{2}}{n
\sqrt{L}} \, \left( L^{-n/2} - (2L)^{\frac{-n+1}{2}}\right)  + n \,
\left( \lambda_0^{\frac{-n-1}{2}} - (2L)^{\frac{-n-1}{2}} \right)
\Biggr).
\end{eqnarray*}
This leads to the conclusion since this right-hand side is bounded
by $16 \, \sqrt{2} \, N_2 \, \lambda_0^{\frac{-n-1}{2}} \, n.$
\end{proof}
\begin{proposition}\label{prop5.5Kais}
For $n \geq 2, n \in \nline$ we have
$$
\left\| \left[\cOL (\lambda^2) \, \lambda^{-n} \right]^{\vee}
\right\|_\infty \leq
 \frac{2\lambda_0^{\frac{-n+1}{2}}}{n-1}.
$$
\end{proposition}
\begin{proof}
As
$$
\left\| \left[ \cOL (\lambda^2) \, \lambda^{-n} \right]^{\vee}
\right\|_\infty \leq \left\|\cOL (\lambda^2) \, \lambda^{-n}
\right\|_1,
$$
we conclude by simple calculations.
\end{proof}
\begin{theorem}\label{thm5.6Kais}
For $n \geq 2$ and $N_2 = \left\|\phi\right\|_{C^2(\rline)}$, it
holds
$$
\left\| \left[\cOL (\lambda^2) \, \lambda^{-n} \right]^{\vee}
\right\|_1 \leq \lambda_0^{-n/2} \, \left( \frac{4}{n-1} + 32
\sqrt{2} \, N_2 \, n \right).
$$
\end{theorem}
\begin{proof}
We   split up the integral in $\rline$  into an integral in $[-
\lambda_0^{- 1/2}, \lambda_0^{- 1/2}]$ and outside, this yields
\begin{eqnarray*}
\left\| \left[\cOL (\lambda^2) \, \lambda^{-n} \right]^{\vee}
\right\|_1 \leq \left\| \left[\cOL (\lambda^2) \, \lambda^{-n}
\right]^{\vee} (\tau) \chi_{[- \lambda_0^{- 1/2}, \lambda_0^{-
1/2}]} (\tau) \right\|_\infty \, \int_{-
\lambda_0^{-1/2}}^{\lambda_0^{- 1/2}} d\tau
\\
+\left\| \left[\cOL (\lambda^2) \, \lambda^{-n} \right]^{\vee}
(\tau) \chi_{\rline \setminus [- \lambda_0^{-1/2},
\lambda_0^{-1/2}]} (\tau) \, \tau^2 \right\|_\infty  \, \int_{\rline
\setminus [- \lambda_0^{-1/2}, \lambda_0^{-1/2}]} \frac{1}{\tau^2}
\, d\tau.
\end{eqnarray*}
The conclusion then follows from Propositions \ref{prop5.4Kais} and
\ref{prop5.5Kais}.
\end{proof}

\begin{proposition}\label{prop5.7Kais}
$$
\left\| \left[\cOL (\lambda^2) \right]^{\vee} \right\|_1 \leq
\left\| \left[\phi(\lambda^2)\right]^{\vee} \right\|_1 + \left\|
\left[\phi(\lambda^2)\right]^{\vee} \right\|_1^2.
$$
\end{proposition}
\begin{proof}
By definition, we have
\begin{eqnarray*}
\left\| \left[\cOL (\lambda^2) \right]^{\vee} \right\|_1 = \left\|
\left[\left(1 - \phi\left(\frac{\lambda^2}{\lambda_0}\right)\right)
\, \phi\left(\frac{\lambda^2}{L}\right) \right]^{\vee}\right\|_1
\\
\leq \left\|\left[\phi\left(\frac{\lambda^2}{L}\right)\right]^{\vee}
\right\|_1
+\left\|\left[\phi\left(\frac{\lambda^2}{\lambda_0}\right)\right]^{\vee}
\, \star \left[\phi \left(\frac{\lambda^2}{L}\right) \right]^{\vee}
\right\|_1
\\
\leq
\left\|\left[\phi\left(\frac{\lambda^2}{L}\right)\right]^{\vee}\right\|_1
+ \left\| \left[ \phi \left(\frac{\lambda^2}{\lambda_0}\right)
\right]^{\vee} \right\|_1 \,
\left\|\left[\phi\left(\frac{\lambda^2}{L}\right)\right]^{\vee}
\right\|_1.
\end{eqnarray*}
For $\alpha \geq 1$, the function $\lambda \mapsto
\phi(\frac{\lambda^2}{\alpha})$ is in $C^\infty (\rline)$ and has
compact support. This justifes the above calculation. The right hand
side of the last inequality is in fact independent of $L$ and
$\lambda_0$, as can be seen as follows: for $\alpha > 0$ we have
$$
\left\|\left[ \phi
\left(\frac{\lambda^2}{\alpha}\right)\right]^{\vee} \right\|_1 =
\sqrt{\alpha} \, \int_{-\infty}^{+\infty}
\left|\left[\phi(\lambda^2)\right]^{\vee} (\sigma) \right| \,
\frac{d\sigma}{\sqrt{\alpha}} =
\left\|\left[\phi(\lambda^2)\right]^{\vee}\right\|_1.
$$
\end{proof}
\begin{proposition}\label{prop5.8Kais}
$$
\left\|\left[\cOL (\lambda^2) \, \lambda^{-1} \right]^{\vee}
\right\|_\infty \leq \left\| \left[\phi(\lambda^2)\right]^{\vee}
\right\|_1 +  \left\| \left[\phi(\lambda^2)\right]^{\vee}
\right\|_1^2.
$$
\end{proposition}
\begin{proof}
We may write
\begin{eqnarray*}
\left\|\left[\cOL (\lambda^2) \, \lambda^{-1} \right]^{\vee}
\right\|_\infty = \left\|\left[\cOL (\lambda^2)\right]^{\vee}\star
\, \left[\lambda^{-1} \right]^{\vee} \right\|_\infty
\\
 \leq \left\|\left[ \cOL (\lambda^2) \right]^{\vee} \right\|_1 \, \left\|\left[ \lambda^{-1}\right]^{\vee} \right\|_\infty.
\end{eqnarray*}
due the fact that $\lambda \mapsto \cOL (\lambda^2)$ is a test
function and $\left[\lambda^{-1}\right]^{\vee} (\tau) = - i \, sign
(\tau), \, \tau \in \rline.$ The conclusion follows from Proposition
\ref{prop5.7Kais}.
\end{proof}
\begin{proposition}\label{prop5.9Kais}
Let $\lambda_0 \geq 1, L \geq 2 \lambda_0.$ Then
$$
\left\|\left[\cOL (\lambda^2) \, \lambda^{-1} \right]^{\vee}
\right\|_1 \leq \left( 2 (N_1 + N_1^2) + 32 \sqrt{2} N_2 \right) \,
\sqrt{\lambda_0},
$$
recalling that $N_1 = \left\| \left[\phi(\lambda^2)\right]^{\vee}
\right\|_1, \, N_2 = \left\|\phi\right\|_{C^2(\rline)}.$
\end{proposition}
\begin{proof}
As before, we write
\begin{eqnarray*}
\left\|\left[\cOL (\lambda^2) \, \lambda^{-1} \right]^{\vee}
\right\|_1 \leq \left\|\left[\cOL (\lambda^2) \, \lambda^{-1}
\right]^{\vee} \chi_{[- \lambda_0^{-1/2},\lambda_0^{-1/2}]}
\right\|_\infty \, \int_{-\lambda_0^{-1/2}}^{\lambda_0^{-1/2}} d\tau
\\+
\left\| \left[ \cOL (\lambda^2) \, \lambda^{-1} \right]^{\vee} \,
\chi_{\rline \setminus [- \lambda_0^{-1/2},\lambda_0^{-1/2}]} (\tau)
\, \tau^2 \right\|_\infty \, \int_{\rline \setminus [-
\lambda_0^{-1/2}, \lambda_0^{-1/2}]} \frac{1}{\tau^2} \, d\tau.
\end{eqnarray*}
We finish the proof by using Propositions \ref{prop5.4Kais} and
\ref{prop5.8Kais}.
\end{proof}


Now we have all the ingredients to state and prove the
$L^\infty-$decay and the perturbation result.

\begin{theorem}\label{high frequency perturbation section}
Let $V,f,g \in L^1({\cal R})$ be real valued,  let $V$ satisfy
the conditions of Theorem \ref{mainresult}, $N \geq 2$,
$0<q_*<1$,
$\lambda_0 > \lambda_* = \frac{4(N-1)^2 \|V\|_1^2}{ N^2 q_*^2}$
and $L > 2 \lambda_0.$ Then we have
\begin{enumerate}
  \item
$ | \langle e^{itH} \cOL (H)f,g \rangle |
\leq
\Bigl(\sum_{k=0}^{\infty} \Bigl(\frac{2(N-1)}{N }\Bigr)^k \|V\|_1^k
\|{\cal F}^{-1}[\cOL(\lambda^2)|\lambda|^{-k}]\|_1\Bigr)
\|f\|_1\|g\|_1 \, |t|^{-1/2}, \\ \hbox{ for }t \neq 0. $
  \item
$
\| e^{itH} \cO (H) \|_{1,\infty}
\leq 4(A + B \frac{\|V\|_1}{\sqrt{\lambda_0}})|t| ^{-1/2} ,
t \neq 0\,,
$ \vspace{3mm}\\
 where
$A= N_1 + N_1^2$ with
$N_1 :=  \|{\cal F}^{-1}[\phi(\lambda^2)]|_1$,
$N_2 :=  \|\phi]|_2$ \\
and
$B= M \frac{(N-1)}{N  }\frac{1}{(1-q_*)^2}$
with
$M := 32 \sqrt{2} \max \{N_1 + N_1^2; N_2\}$,
  \item
$
\| e^{itH} \cO (H) -  e^{itH_0} \cO (H_0)\|_{1,\infty}
\leq 4B \frac{\|V\|_1}{\sqrt{\lambda_0}}|t| ^{-1/2} ,
t \neq 0 \, .
$  \vspace{3mm}\\
with $B$ as in 2.
In particular we have
\[
e^{itH} \cO (H)f \rightarrow e^{itH_0} \cO (H_0)f \hbox{ for }
\lambda_0 \rightarrow \infty
\]
uniformly on ${\cal R}$ for every fixed $t>0$ or also uniformly on
${\cal R} \times [\epsilon, \infty ) $ with respect to the weight
$|t|^2$ on the time axis for any positive $\epsilon$.
\end{enumerate}
\end{theorem}
\begin{proof}~
\\
\underline{1.:}
\\
%
At first we consider $f \in L^1({\cal R}) \cap L^2({\cal R})$, the
estimates then extend to $f \in L^1({\cal R})$.
\\
From Stone's formula, the fact that the spectrum of $H$ is
absolutely continuous on $[0,\infty)$ (Remark \ref{type of
spectrum}) and the limiting absorption principle proved in Theorem
\ref{lim.abs} we deduce
\[
\langle e^{itH} \cOL(H) f,g \rangle= \frac{1}{2i\pi}\int_0^\infty
e^{it\lambda } \cOL(\lambda ) \langle
(R_V(\lambda+i0)-R_V(\lambda-i0)) f,g \rangle d\lambda.
\]
As $V$, $f$ and $g$ are real valued, we obtain
\[
\langle e^{itH} \cOL(H) f,g \rangle= \frac{1}{\pi}\int_0^\infty
e^{it\lambda } \cOL(\lambda ) \Im\langle R_V(\lambda+i0) f,g \rangle
d\lambda.
\]
Using Proposition \ref{Born series} part 3. and the change of
variables $\lambda=\mu^2$ we find
\[
\langle e^{itH} \cOL(H) f,g \rangle= \frac{2}{\pi}\int_0^\infty
e^{it\mu^2} \cOL(\mu^2 ) \sum_{k=0}^\infty \Im\langle R_0(\mu^2+i0)
(-V R_0(\mu^2+i0))^k f,g \rangle \mu d\mu.
\]
Fubini's Theorem, whose hypotheses are fulfilled thanks to the
inequality in the proof of Proposition \ref{Born series} part 3.,
leads to
\begin{eqnarray*}
\langle e^{itH} \cOL(H) f,g \rangle= \frac{2}{\pi} \sum_{k=0}^\infty
\int_{\caR} \int_{\caR^k}  \prod_{j=1}^k V(x_j)\int_{\caR}
\\
\Big(\int_0^\infty e^{it\mu^2} \cOL(\mu^2 ) N(x,x_1,\cdots, x_k, y,
\mu)
 \mu d\mu  \Big)\, f(y) \,dy   \,dx_1 \cdots \,dx_k  g(x) \,dx,
\end{eqnarray*}
where $N(x,x_1,\cdots, x_k, y, \mu)$ is defined by
\[
N(x,x_1,\cdots, x_k, y, \mu)=(-1)^k \Im \Big(K_0(x,x_1,\mu^2+i0)
\prod_{j=1}^{k-1} K_0(x_j,x_{j+1},\mu^2+i0)
K_0(x_k,y,\mu^2+i0)\Big).
\]
Using again Proposition \ref{Born series}, after some elementary
calculations, we find that
\[
N(x,x_1,\cdots, x_k, y, \mu)= -\frac{1}{\mu^{k+1} N^{k+1}}
e^{\frac{ik\pi}{2}} \sum_{n=1}^{2^k}(1-\frac{N}{2})^{\alpha_n}
(\frac{N}{2})^{\beta_n}  (e^{ i d_n \mu} +(-1)^k  e^{- i d_n \mu}),
\]
with $\alpha_n, \delta_n\in \nline$ such that $\alpha_n+
\delta_n=k+1$ and $d_n$ are real numbers that depend on $x, y$ and
$x_j$, $j=1,\cdots, k$. Using this expression in the previous one,
we obtain
\begin{eqnarray*}
\langle e^{itH} \cOL(H) f,g \rangle= -\frac{1}{\pi}
\sum_{k=0}^\infty  N^{-(k+1)} e^{\frac{ik\pi}{2}} \int_{\caR}
\int_{\caR^k}  \prod_{j=1}^k V(x_j) \int_{\caR}
 \sum_{n=1}^{2^k}(1-\frac{N}{2})^{\alpha_n} (\frac{N}{2})^{\beta_n}
\\
\Big(\int_{-\infty}^\infty e^{it\mu^2} \cOL(\mu^2 )
 e^{ i d_n \mu}
 \mu^{-k} d\mu \Big) f(y)\, dy   \,dx_1 \cdots \,dx_k  g(x) \,dx.
\end{eqnarray*}
Noting that $|(1-\frac{N}{2})^{\alpha_n}
(\frac{N}{2})^{\beta_n}|=(\frac{N}{2}-1)^{\alpha_n}
(\frac{N}{2})^{\beta_n}\leq (N-1)^{\alpha_n}
(N-1)^{\beta_n}=(N-1)^{k+1}$, we find

\begin{eqnarray*}
|\langle e^{itH} \cOL(H) f,g \rangle|\leq \frac{1}{\pi}
\sum_{k=0}^\infty  \frac{(N-1)^{k+1}}{N^{k+1}} \int_{\caR}
\int_{\caR^k}  \prod_{j=1}^k |V(x_j)|
\\ \int_{\caR}
 \sum_{n=1}^{2^k}
\Big|\int_{-\infty}^\infty e^{it\mu^2} \cOL(\mu^2 )   e^{ i d_n \mu}
 \mu^{-k} d\mu\Big|
  |f(y)| \, dy   \,dx_1 \cdots \,dx_k  |g(x)| \,dx.
\end{eqnarray*}
Setting
\[
S_k=\sup_{d\in \rline}  \Big|\int_0^\infty e^{i(t\mu^2+d \mu)}
\cOL(\mu^2 )
 \mu^{-k} d\mu\Big|,
\]
we deduce that
\begin{eqnarray*}
|\langle e^{itH} \cOL(H) f,g \rangle|\leq \frac{1}{\pi} \|f\|_1
\|g\|_1 \sum_{k=0}^\infty  \frac{2^{k+2}(N-1)^{k+1}}{N^{k+1}}
   \|V\|_1^k S_k.
\end{eqnarray*}
%

%
We observe that
\[
S_k = \sup_{a \in \rline} \left| \int_{- \infty}^{+ \infty} e^{i(t
\lambda^2 + a \lambda)} \cOL(\lambda^2) \lambda^{-k}
 d \lambda \right|
\leq
\|{\cal F}^{-1}\big[\cOL(\lambda^2)\lambda^{-k}\big]\|_1
 \, |t|^{-1/2}, t \neq 0 \, ,
\]
since the quantity inside the absolute value  is the solution of the
free Schr\"odinger operator on $\RR$ at time $t$ and position $a$
for the initial condition
${\cal F}^{-1}\big[\cOL(\lambda^2)\lambda^{-k}\big]$,
see for example \cite{ReedSimonII}, p. 60 Theorem IX.30. The
convergence of the series will follow from the proof of 2.
%
\newpage
\noindent \underline{2.:}
\\
First let $f,g \in L^1({\cal R})$ be real valued.  With $\tilde
q(\lambda) := \frac{2(N-1)\|V\|_1}{N\sqrt{\lambda}}$ and the
assumptions $0<q_*<1$ and $\lambda_* \leq \lambda_0$ it follows $0 <
\tilde q(\lambda_0) < \tilde q(\lambda_*) = q_*< 1.$ Therefore
\[
\sum_{k=1}^\infty \tilde q(\lambda_0)^k k =
\frac{1}{(1-\tilde q(\lambda_0))^2}
\]
converges. Thus we can apply Theorem \ref{p1} and obtain together
with 1.
\begin{eqnarray*}
  | \langle e^{itH} \cO (H)f,g \rangle |& \leq &
\Bigl(\sum_{k=0}^{\infty} \Bigl(\frac{2(N-1)}{N }\Bigr)^k \|V\|_1^k
\, c(k) \lambda_0^{-k/2}\Bigr)
\|f\|_1\|g\|_1 \, |t|^{-1/2}
   \\
   &=&
\Bigl(\sum_{k=0}^{\infty} \tilde q(\lambda_0)^k  c(k) \Bigr)
\|f\|_1\|g\|_1 \, |t|^{-1/2}
   \\
   &=&
\Bigl(c(0) + \sum_{k=1}^{\infty} \tilde q(\lambda_0)^k  c(k) \Bigr)
\|f\|_1\|g\|_1 \, |t|^{-1/2}
   \\
   &\leq&
\Bigl(c(0) +  M \sum_{k=1}^{\infty} \tilde q(\lambda_0)^k k \Bigr)
\|f\|_1\|g\|_1 \, |t|^{-1/2}
   \\
   &=&
\Bigl(c(0) +  M \frac{1}{(1-\tilde q(\lambda_0))^2} \Bigr)
\|f\|_1\|g\|_1 \, |t|^{-1/2}
\\
   &\leq&
\Bigl(c(0) +  M \frac{1}{(1-q_*)^2} \Bigr)
\|f\|_1\|g\|_1 \, |t|^{-1/2}
\end{eqnarray*}
for all  $t \neq 0$.
Since spectral measures are finite and since
$\lim_{L\rightarrow \infty} \cOL = \cO$ pointwise,
we can replace $\cOL$ by $\cO$ in the last inequality using
dominated convergence. By linearity a factor $4$ appears for complex
valued $f$ and $g.$ This ends the proof of 2.
\\
\underline{3.:}
\\
Stone's formula applied to $H_0$ yields
\[
\langle e^{itH_0} \cOL(H_0) f,g \rangle= \frac{1}{\pi}\int_0^\infty
e^{it\lambda } \cOL(\lambda ) \Im\langle R_0(\lambda+i0) f,g \rangle
d\lambda,
\]
which is the first term in the expansion for $\langle e^{itH}
\cOL(H) f,g \rangle.$ Therefore
\[
\langle (e^{itH} \cOL(H) - e^{itH_0} \cOL(H_0))f,g \rangle=
\frac{2}{\pi}\int_0^\infty e^{it\mu^2} \cOL(\mu^2 )
\sum_{k=1}^\infty \Im\langle R_0(\mu^2+i0) (-V R_0(\mu^2+i0))^k f,g
\rangle \mu d\mu.
\]
Now the same proof as in 2. but without the first term yields the
assertion.
\end{proof}


\subsection{Low energy estimate}

In this section we consider the case of initial conditions with
compact energy band with respect to the spectral representation.
Again we adapt the reasoning of \cite{GS} to the transmission
situation.

For any smooth and compactly supported cut-off function $\chi$ in $\RR$,
by Theorem \ref{res.id}
we have for any $x, x'\in R$
$$
2i\pi\int_0^{+ \infty} e^{it\lambda} \chi(\lambda) E_{ac}(d\lambda)(x,x')=
-2i\pi\int_0^{+\infty} e^{it\lambda} \chi(\lambda) \Im K(x,x',\lambda) d\lambda,
$$
and by the change of variables $\lambda=\mu^2$, we get
$$
2i\pi\int_0^{+\infty} e^{it\lambda} \chi(\lambda) E_{ac}(d\lambda)(x,x')=
-4i\pi\int_0^{+\infty} e^{it\mu^2} \chi(\mu^2) \Im K(x,x',\mu^2) \mu  d\mu.
$$
Now recalling the definition of $K$, we again distinguish between the following three cases:
\begin{enumerate}
\item
If $x, x'\in R_j$ with $x'>x$, then
$$
K(x,x',\mu^2)=\frac{1}{W_{j,-}(\mu)}  f_{j,-}(\mu,x)f_{j,+}(\mu,x')
+\frac{c_{j,-,2}(\mu) }{i\mu f_{j+1,+}(\mu,0) s(\mu)}
 f_{j,+}(\mu,x)  f_{j,+}(\mu,x').
$$
As $\overline{f_{j,\pm}(\mu,x)}=f_{j,\pm}(-\mu,x)$, we deduce
\beqs
2i\pi\int_0^{+\infty} e^{it\lambda} \chi(\lambda) E_{ac}(d\lambda)(x,x')&=&
-2i\pi\int_{-\infty}^{+\infty} e^{it\mu^2} \chi(\mu^2)\mu    \frac{1}{W_{j,-}(\mu)}  f_{j,-}(\mu,x)f_{j,+}(\mu,x')  d\mu
\\
&-&
2\pi\int_{-\infty}^{+\infty} e^{it\mu^2} \chi(\mu^2)   \frac{c_{j,-,2}(\mu) }{f_{j+1,+}(\mu,0) s(\mu)}
 f_{j,+}(\mu,x)  f_{j,+}(\mu,x') d\mu.
\eeqs

The first term of this right hand side was estimated in Lemma 4 of \cite{GS}, hence it remains to estimate the second term. For that purpose, we set
\beqs
T_2(t,x,x')&:=&\int_{-\infty}^{+\infty} e^{it\mu^2} \chi(\mu^2)   \frac{c_{j,-,2}(\mu) }{f_{j+1,+}(\mu,0) s(\mu)}
 f_{j,+}(\mu,x)  f_{j,+}(\mu,x') d\mu\\
&=&\int_{-\infty}^{+\infty} e^{it\mu^2} \chi(\mu^2) e^{i\mu (x+x')}  \frac{c_{j,-,2}(\mu) }{f_{j+1,+}(\mu,0) s(\mu)}
 m_{j,+}(\mu,x)  m_{j,+}(\mu,x') d\mu.
\eeqs
Hence denoting by
$$
p(\mu):=\frac{c_{j,-,2}(\mu) }{f_{j+1,+}(\mu,0) s(\mu)},
$$
we have shown in Corollary \ref{corouncoefdansH1} that this function belongs to
$H^1(-R,R)$, for all $R>0$. Since the mapping
$$
q:\mu\to \chi(\mu^2) m_{j,+}(\mu,x) m_{j,+}(\mu,x'),
$$
has compact support and is in $C^1(\RR)$ with the property
$$
|q(\mu)|+|\dot{q}(\mu)|\leq C,
$$
for some $C>0$ independent of $x$ and $x'$ due to (\ref{estm+}) and
(\ref{estmdot}), we deduce that the product $pq$ belongs to
$H^1(\RR)$. By Plancherel theorem (see for instance \cite[p.
60]{ReedSimonII}), we deduce that
$$
T_2(t,x,x')=t^{-\frac12} \int_{-\infty}^{+\infty} {\mathcal F}^{-1}(pq)(\xi+x+x') e^{-\frac{i\xi^2}{t}}d\xi.
$$
and consequently
\beqs
|T_2(t,x,x')|&\leq& |t|^{-\frac12} \int_{-\infty}^{+\infty} |{\mathcal F}^{-1}(pq)(\xi+x+x')|d\xi
\\
&\leq& |t|^{-\frac12} \int_{-\infty}^{+\infty} |{\mathcal F}^{-1}(pq)(\xi)|d\xi
\\
&\leq& C |t|^{-\frac12} \|pq\|_{H^1(\RR)},
\eeqs
for some $C>0$.
\item
If $x, x'\in R_j$ with $x'<x$, then
$$
K(x,x',\mu^2)=\frac{1}{W_{j,-}(\mu)}  f_{j,-}(\mu,x')f_{j,+}(\mu,x)
+\frac{c_{j,-,2}(\mu) }{i\mu f_{j+1,+}(\mu,0) s(\mu)}
 f_{j,+}(\mu,x)  f_{j,+}(\mu,x').
$$

In that case the first term  was treated  in  Lemma 4 of \cite{GS}, while the second term is the same as
before.
\item
If $x\in R_j$ and $x'\in R_k$ with $k\ne j$, then
$$
K(x,x',\mu^2)= \frac{1}{i\mu f_{j+1,+}(\mu,0) s(\mu)}
 f_{j,+}(\mu,x)  f_{k,+}(\mu,x').
$$
Therefore in that case we have
$$
2i\pi\int_0^{+\infty} e^{it\lambda} \chi(\lambda) E_{ac}(d\lambda)(x,x')=2\pi
\int_{-\infty}^{+\infty} e^{it\mu^2} \chi(\mu^2)   \frac{1}{f_{j+1,+}(\mu,0) s(\mu)}
 f_{j,+}(\mu,x)  f_{j,+}(\mu,x') d\mu.
$$
Since $\frac{1}{f_{j+1,+}(\mu,0)}$ is in $C^1(\RR)$, by Corollary \ref{corosdansH1},
the function
$$
\mu\to  \frac{1}{f_{j+1,+}(\mu,0) s(\mu)}
$$
belongs to $H^1(-R,R)$, for all $R>0$ and we conclude as for $T_2(t,x,x')$.

\qed

\end{enumerate}


\end{document}